\documentclass[10pt,reqno]{amsart}

\usepackage{amscd}
\usepackage{amsthm}
\usepackage{amssymb}

\usepackage{enumerate,array,mathrsfs, a4wide}
\usepackage{graphicx}
\usepackage{sseq}
\usepackage{xcolor}

\usepackage{rotating}
\usepackage{bbm}
\usepackage[all,cmtip]{xy}

\usepackage[latin1]{inputenc}
\usepackage{dsfont}

\usepackage{hyperref}

\newtheorem{theorem}{Theorem}[section]

\newtheorem{lemma}[theorem]{Lemma}
\newtheorem{proposition}[theorem]{Proposition}
\newtheorem{corollary}[theorem]{Corollary}

\newtheorem*{theorem*}{Theorem}

\theoremstyle{definition}
\newtheorem{definition}[theorem]{Definition}

\newtheorem{example}[theorem]{Example}

\theoremstyle{remark}
\newtheorem{remark}[theorem]{Remark}

\newenvironment{setupS}

\numberwithin{equation}{section}

\newcommand{\abs}[1]{\lvert#1\rvert}

\newcommand{\ow}{\omega}

\newcommand{\C}{{\mathbb{C}}}
\newcommand{\R}{{\mathbb{R}}}
\newcommand{\Q}{{\mathbb{Q}}}

\newcommand{\Z}{{\mathbb{Z}}}
\newcommand{\N}{{\mathbb{N}}}

\renewcommand{\epsilon}{\varepsilon}
\renewcommand{\phi}{\varphi}
\renewcommand{\theta}{\vartheta}

\newcommand{\lie}[1]{{\mathcal{L}_{#1}}}

\newcommand{\w}{\wedge}

\DeclareMathOperator{\id}{Id}

\DeclareMathOperator{\OB}{OB}

\DeclareMathOperator{\sgn}{sgn}

\DeclareMathOperator{\Spec}{Spec}

\DeclareMathOperator{\open}{OB}

\begin{document}
\title[Open books for Boothby--Wang bundles]{Open books for Boothby--Wang bundles, fibered Dehn twists and the mean Euler characteristic}

\author[River Chiang]{River Chiang}
\address{Department of Mathematics, National Cheng Kung University,
Tainan 701, Taiwan
}
\email{riverch@mail.ncku.edu.tw}

\author[Fan Ding]{Fan Ding}
\address{School of Mathematical Sciences and LMAM, Peking University, Beijing 100871, P.~R.~China
}
\email{dingfan@math.pku.edu.cn}

\author[Otto van Koert]{Otto van Koert}

\address{
Department of Mathematical Sciences, Seoul National University\\
Building 27, room 402\\
San 56-1, Sillim-dong, Gwanak-gu, Seoul, South Korea\\
Postal code 151-747
}
\email{okoert@snu.ac.kr}




\begin{abstract}
We examine open books with powers of fibered Dehn twists as monodromy. The resulting contact manifolds can be thought of as Boothby--Wang orbibundles over symplectic orbifolds. Using the mean Euler characteristic of equivariant symplectic homology we can distinguish these contact manifolds and hence show that some fibered Dehn twists are not symplectically isotopic to the identity relative to the boundary. This complements results of Biran and Giroux.
\end{abstract}

\maketitle

\section{Introduction}
Since Giroux established the correspondence between open books with symplectic monodromy and contact manifolds, there has been a lot of activity to investigate this relation further. 
In dimension 3, this approach has been particularly fruitful, since the requirement that the monodromy is a symplectomorphism imposes no real constraints; it is possible to use the wealth of knowledge about the mapping class groups of surfaces.
One can think of a mapping class group of a surface as being generated by Dehn twists along curves. 

For a general symplectic manifold, the symplectomorphism group is not understood very well. 
Nevertheless, let us mention here the result of Seidel, \cite{{Seidel:TS2}}, on the compactly supported symplectomorphism group of $T^*S^n$ with its canonical symplectic structure: for $n=2$, the generalized Dehn twists generate the group;
for $n>2$, Dehn twists form an infinite cyclic subgroup.  
The latter result can be recovered, through the Giroux correspondence, by considering the associated open books with page $T^*S^n$ and $N$-fold right-handed Dehn twist as monodromy. We shall denote these manifolds by $\open(T^*S^n,\tau^N)$.
In~\cite{vanKoert:brieskorn_openbook} it was shown that these contact manifolds are contactomorphic to Brieskorn manifolds, $\open(T^*S^n,\tau^N)\cong \Sigma(N,2,\ldots,2)$. 
The contact structures on these manifolds can be distinguished using the mean Euler characteristic of equivariant symplectic homology, see Section~\ref{sec:s1eq} for the definition of this notion. 
For nice contact manifolds, including these Brieskorn manifolds, one can compute the mean Euler characteristic completely in terms of Reeb orbit data; Floer theory is necessary but only to show invariance of this number. 
For the Brieskorn manifold $\Sigma(N,2,\ldots,2)$ of dimension $2n+1$ (with $n$ even and $N$ odd), the mean Euler characteristic is $\chi_m(\Sigma(N,2,\ldots,2)\,)=\frac{1}{2}\frac{nN+1}{(n-1)N+2}$, see for example~\cite{Ginzburg:homological_resonances} or~\cite{vanKoert:thesis}.
This is an injective function of $N$. It implies that all odd powers of Dehn twists $\tau^N$ are distinct, and so are all powers of Dehn twists $\tau^N$. 
Note that for $n$ odd, the Brieskorn manifolds $\Sigma(N, 2, \ldots,2)$ are all non-diffeomorphic.   

Furthermore, such Dehn twists can be constructed for any symplectic manifold containing a Lagrangian sphere. 
However, if a symplectic manifold does not contain Lagrangian spheres, there is no general procedure to construct symplectomorphisms that are not symplectically isotopic to the identity.

On the other hand, Biran and Giroux~\cite{Biran_Giroux:fibered_Dehn} considered the case of fibered Dehn twists, which can be constructed if the contact type boundary of a symplectic manifold admits a suitable $S^1$-action.
More precisely, consider a symplectic manifold $W$ with contact type hypersurface $P$ carrying a free $S^1$-action in the neighborhood $P\times [0,1]$ that preserves the contact form on $P$. 
Then one can define a right-handed fibered Dehn twist as a map of the form
\begin{align*}
\tau: P\times [0,1] & \longrightarrow  P\times [0,1] \\
(x,t) & \longmapsto  (x\cdot\left[ f(t) \,\mathrm{mod}\, 2\pi \right],t)
\end{align*}
by choosing a function $f: [0,1] \to \R$ such that $f$ equals $2\pi$ near $t=0$ and $0$ near $t=1$. 
The map $\tau$ is a symplectomorphism that is the identity near the boundary of $P\times [0,1]$.
This allows one to extend $\tau$ to a symplectomorphism of $W$. 

Biran and Giroux~\cite{Biran_Giroux:fibered_Dehn} showed that such fibered Dehn twists are often not symplectically isotopic to the identity.

\begin{theorem*}[Biran and Giroux]
	Let $(M^{2n},\omega)$ be an integral symplectic manifold with an adapted Donaldson hypersurface $H$ that is Poincar\'e dual to $[\omega]$. Consider $W=M-\nu(H)$, 
the complement of a tubular neighborhood of $H$.
Suppose that one of the following conditions hold
\begin{enumerate}
\item $\pi_2(M)=0$, 
\item $M$ is monotone and contains a simply connected Lagrangian such that
the minimal Chern number $c_M$ of $M$ satisfies $c_M \geq (n+2)/2$.
\end{enumerate}
Then a right-handed fibered Dehn twist $\tau$ on $W$ along the boundary $\partial W$ is not symplectically isotopic to the identity relative to the boundary. 
\end{theorem*}

Their proof used Lagrangian Floer homology. 
First they observed that certain Boothby--Wang bundles over symplectic manifolds carry a supporting open book whose monodromy is a fibered Dehn twist.
By a result of Cieliebak, see Theorem~\ref{thm:subcritical_open_book}, contact open books with trivial monodromy are always subcritically Stein fillable.
On the other hand, compact sets in subcritical Stein manifolds are Hamiltonian displaceable, so Lagrangian Floer homology must be trivial. The right sets of assumptions
guarantee nontrivial Lagrangian Floer homology, so one can deduce in this way that these fibered Dehn twists are not symplectically isotopic to the identity.

We shall also address this question, but use a different approach.
The main idea is that the Reeb dynamics in subcritical manifolds are fairly well understood, and this gives similar but different conditions for triviality of fibered 
Dehn twists. We shall also consider powers of fibered Dehn twists. 
Let us begin by stating the following result (Theorem~\ref{thm:Boothby_Wang_orbifold_open_book}).
\begin{theorem*}
\label{thm:OB_Boothby-Wang}
Let $(M,\omega)$ be an integral symplectic manifold with an adapted Donaldson hypersurface $H$ that is Poincar\'e dual to $[\omega]$. Consider $W=M-\nu(H)$, 
the complement of a tubular neighborhood of $H$. 
Let $\tau$ denote a right-handed fibered Dehn twist on $W$ along the boundary $\partial W$. Then for any positive integer $N$, $\open(W,\tau^N)$
carries the structure of a Boothby--Wang orbibundle over a symplectic orbifold.
\end{theorem*}

Note that the special case of $k=1$ recovers the open book decomposition of Boothby--Wang bundles considered by Biran and Giroux.

The methods we shall use to distinguish fibered Dehn twists from the identity are the following. 
First of all, there always exist contractible Reeb orbits in subcritically fillable contact manifolds. This is not always the case for the above class of Boothby--Wang orbibundles.
Secondly, the mean index, i.e.~the ``average'' Maslov index of periodic Reeb orbits, is positive in subcritically fillable contact manifolds if the first Chern class is trivial. Many Boothby--Wang orbibundles have negative mean index though.  
A related statement was made by Oancea and Viterbo, \cite{OV}, Proposition~5.14. 
Thirdly, we can use the mean Euler characteristic of equivariant symplectic homology as mentioned earlier: for exactly fillable contact manifolds this number can be thought of as a contact invariant. Moreover, this number has to be a half-integer for subcritically fillable contact manifolds.

The main application of our methods is the following result (Theorem~\ref{thm:distinguishing_powers}), while we also provide a short proof of the first case of the above-mentioned theorem of Biran and Giroux. See Theorem \ref{thm:pi_2_fibered_twist_not_isotopic_id}.

\begin{theorem*}
Let $(M^{2n-2},\omega)$ be a simply connected symplectic manifold of dimension at least $6$ such that $[\omega]\in H^2(M;\Z)$ is a primitive element.
Suppose that $c_1(M)=c[\omega]$, and let $H$ be an adapted Donaldson hypersurface that is Poincar\'e dual to $k[\omega]$ for some positive integer $k$.
Let $\tau$ denote a right-handed fibered Dehn twist along the boundary of $M-\nu(H)$.
If $\tau^N$ is symplectically isotopic to the identity relative to the boundary for a positive integer $N$, then one of the following conditions must hold,
\begin{itemize}
\item $c\geq k$, $k$ does not divide $N$, and $\chi(H)=\chi(M)=0$.
\item $c=k$, $k$ divides $N$, and $\chi(H)=0$.
\item $c> k$, $k$ divides $N$, and $\left((c-k)k+1 \right) \chi(H)=(c-k)k\chi(M)$.
\end{itemize} 
\end{theorem*}

In many cases this means that all positive powers of fibered Dehn twists are distinct. 
We illustrate this with examples of certain smooth complete intersections, see 
Examples \ref{example:CP-H_d} and \ref{example:degree_d_hypersurface}.

Finally, we want mention an explicit formula for the mean Euler characteristic, which might be of independent interest. With the notation from our first theorem, put $P_N=\OB(W^{2n-2},\tau^N)$. 
Then we have the following formula in case that $k=1$,
$$
\chi_m(P_N)=(-1)^{n+1}\frac{(N-1)\chi(H)+\chi(M)}{2|(c-1)N+1|}.
$$
Observing that $M_N:=P_N/S^1$ carries the structure of a symplectic orbifold, we can also write 
$$
\chi_m(P_N)=(-1)^{n+1}\frac{\chi(\vert \w M_N\vert )}{2N|\langle c_1^{orb}(M_N),[B_N] \rangle|},
$$
where
\begin{itemize}
\item $N$ is the total number of sectors,
\item $\w M_N$ is the inertia orbifold associated with $M_N$, and $\chi(\vert \w M_N\vert )$ is the Euler characteristic of the topological space associated with this orbifold,
\item $c_1^{orb}(M_N)$ is the orbifold Chern class, and 
\item $[B_N]$ is the homology class of an (orbi)-sphere intersecting $H$ transversely in one point: this is made precise in Proposition~\ref{prop:mean_euler_boothby_wang_orbifold}.
\end{itemize}
See Proposition~\ref{prop:mean_euler_boothby_wang_orbifold} for more details.

\subsection*{Plan of the paper}
The paper is organized as follows. In Section~\ref{sec:prelim}, we give the basic definitions. In Section~\ref{sec:maslov_index}, we review the notion of Maslov index. In Section~\ref{sec:s1eq}, we discuss $S^1$-equivariant symplectic homology and its mean Euler characteristic, in order to have a suitable invariant. In Section~\ref{sec:bw}, we discuss the conditions for a Boothby--Wang bundle to possess a supporting open book. In Section~\ref{sec:abs open book}, we construct the contact open book from the data we extracted from a Boothby--Wang bundle with fibered Dehn twists as monodromy. In Section~\ref{sec:applications}, we apply our construction and inspect Reeb dynamics to distinguish fibered Dehn twists, and conclude our paper with a discussion about fibered Dehn twists that are not smoothly isotopic to the identity relative to the boundary.

\subsection*{Acknowledgements}
RC is partially supported by the NSC grant 100-2115-M-006-002 and NCTS(South), Taiwan;
FD by grant no 10631060 of the National Natural Science Foundation of China;
and OvK by the NRF Grant 2012-011755 funded by the Korean government.

\section{Preliminaries}
\label{sec:prelim}

\subsection{Boothby--Wang or prequantization circle bundles}

Let $(M,\omega)$ be a compact symplectic manifold with integral
symplectic form, i.e.~$[\omega]\in H^2(M;\Z)$. Then there is a
unique (up to isomorphism) differentiable complex line bundle $L$
over $M$ with $c_1(L)=[\omega]$.

Its associated principal $S^1$-bundle $\Pi: P \to M$ carries a
contact form $\theta$, the so-called {\bf Boothby--Wang form},
which is a connection $1$-form on $P$ with curvature form
\[
d\theta = -2\pi\Pi^*\ow.
\]
The condition that $\theta$ is a connection form means that
the vector field $R_\theta$ generating the principal $S^1$-action satisfies the following equations,
$$
\iota_{R_\theta} \theta=1,\quad\iota_{R_\theta}d\theta=0.
$$
It is therefore the Reeb vector field for $\theta$. This
$S^1$-bundle is called a \textbf{Boothby--Wang bundle} associated
with $(M,\ow)$. It is also known as a prequantization circle
bundle.

\subsection{Weinstein manifolds}


\begin{definition}
  Let $(W,\omega)$ be a symplectic manifold. A proper smooth function $f:\, W
  \to [0,\infty[$ is called \textbf{$\omega$--convex} if it admits a
  complete gradient-like Liouville vector field $X$, i.e.~$\lie{X}
  \omega = \omega$.  We say $(W,\omega)$ is a \textbf{Weinstein
    manifold} if there exists an $\omega$--convex Morse function.
We say the Weinstein manifold is of {\bf finite type} if, in addition, the $\omega$--convex Morse function has only finitely many critical points.
\end{definition}

\begin{remark}
From this definition it follows that all ends of a finite type Weinstein manifold $W$ are convex, i.e.~they look like symplectizations.
We shall often abuse notation, and just use the word Weinstein manifold to mean Weinstein manifold of finite type.
\end{remark}

Note that $\iota_X \omega$ defines a primitive of $\omega$, so Weinstein
manifolds are exact symplectic.
\begin{remark}
  A \textbf{compact Weinstein manifold} or \textbf{Weinstein domain} $(W_0,\omega)$ is a compact
  symplectic manifold with boundary that can be embedded into a
  Weinstein manifold $(W,\omega)$ with an $\omega$--convex Morse function
  $f$ such that $W_0$ is given as the preimage $f^{-1}([0,C])$, and such that $C=\partial W_0$ is a regular value of $f$.
  Note that the corresponding regular level set is automatically contact.
\end{remark}

In practice $\omega$--convex functions can often be found by
looking for strictly plurisubharmonic functions. Recall here that,
for a complex manifold $(W,J)$, a smooth function $f:\, W \to \R$
is \textbf{strictly plurisubharmonic} if $g(X,Y) := -d(df \circ
J)(X, JY)$ defines a Riemannian metric. Stein manifolds can be
defined as complex manifolds admitting exhausting, strictly
plurisubharmonic functions. A compact complex manifold $W_0$ with
boundary is called a compact Stein manifold if it admits a
strictly plurisubharmonic function $f$ such that the boundary
$\partial W_0$ is a regular level set. A compact Stein manifold is
a compact Weinstein manifold. By results of Eliashberg, Weinstein
manifolds can be deformed into Stein manifolds.

\subsection{Contact open books}
\label{sec:contact_open_books}

\begin{definition}
An \textbf{abstract (contact) open book} $(W, \lambda, \psi)$ consists of 
\begin{itemize}
\item a compact Weinstein manifold $(W, d\lambda)$ with $\lambda$ being a primitive of its symplectic form such that the Liouville vector field $X$ for $d\lambda$ defined by $\iota_X d\lambda=\lambda$ is transverse to $\partial W$ and pointing outward, and
\item a symplectomorphism $\psi: W \to W$ equal to the identity near the boundary $\partial W$.
\end{itemize}
\end{definition}

Given an abstract (contact) open book, Giroux proposed an explicit construction of a closed contact manifold.
This construction is as follows.
First we assume $\psi^*\lambda =\lambda-dh$ where $h$ is a positive function.
This can always be done by the following lemma \cite{Giroux:talk}.
\begin{lemma}[Giroux]
\label{lemma:deform_symplectomorphism} The symplectomorphism
$\psi$ can be isotoped, via symplectomorphisms equal to the
identity near $\partial W$, to a symplectomorphism $\widehat\psi$
that satisfies $\widehat\psi^*\lambda = \lambda-dh$.
\end{lemma}

Then we define
\begin{equation*}
  A_{(W,\psi)} := W \times \R / (x,\phi)\sim (\psi(x),\phi+h(x)) \;.
\end{equation*}
This mapping torus carries the contact form
$$
\alpha = d\phi +\lambda.
$$
Since $\psi$ is the identity near the boundary of $W$, a
neighborhood of the boundary looks like $\partial W \times\,
]{-\varepsilon},0] \times  S^1$, with contact form $\alpha =
Cd\phi+ e^t\,\lambda|_{\partial W}$. Here $\phi\in S^1=\R/2\pi\Z$,
$t\in ]{-\varepsilon},0]$ and $C>0,\, 0<\varepsilon<1$ are
constant. Denote the disk $\bigl\{z\in \C\mid \abs{z} < R\bigr\}$
by $D_R^2$ and the annulus $\bigl\{z\in \C\mid r < \abs{z} <
R\bigr\}$ by $A(r,R)$. The closed unit disk $\bigl\{z\in \C\mid
\abs{z} \le 1\bigr\}$ is denoted by $D^2$. We can glue the mapping
torus $A_{(W,\psi)}$ to
\[
  B_W:=\partial W \times D^2_{1+\varepsilon}
\]
using the map
\[
  \begin{split}
    \Phi_{\mathrm{glue}}:\, \partial W \times A(1,1+\varepsilon)
    & \longrightarrow \partial W \times\, ]{-\varepsilon},0] \times S^1 \\
    \bigl(x, re^{i\phi}\bigr) &\longmapsto \bigl(x, 1-r, \phi\bigr)
    \;.
  \end{split}
\]
Pulling back the form $\alpha$ by $\Phi_{\mathrm{glue}}$, we
obtain $Cd\phi+e^{1-r}\, \lambda|_{\partial W}$ on $\partial W
\times A(1,1+\varepsilon)$, which can be easily extended to a
contact form
\[
  \beta = h_1(r)\, \lambda|_{\partial W} + h_2(r)\, d\phi
\]
on $B_W$ by requiring that $h_1$ and $h_2$ are functions from
$[0,1]$ to $\R$ whose behavior is indicated in
Figure~\ref{fig:functions_binding}: $h_1(r)$ has
exponential drop-off and $h_2(r)$ increases quadratically 
near $0$ and is constant near $1$.

\begin{figure}[htp]
\def\svgwidth{0.5\textwidth}%
\begingroup\endlinechar=-1
\resizebox{0.5\textwidth}{!}{%
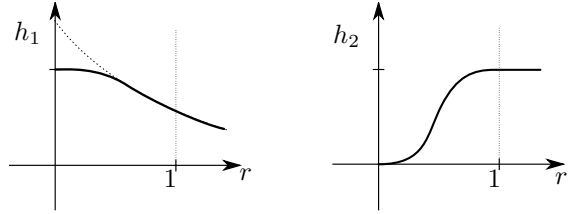%
}\endgroup
\caption{Functions for the contact form near the binding}
\label{fig:functions_binding}
\end{figure}

Gluing $A_{(W,\psi)}$ to $B_W$ via $\Phi_{\mathrm{glue}}$, we get
a closed manifold $M$. Note that the contact forms $\alpha$ on
$A_{(W,\psi)}$ and $\beta$ on $B_W$ glue together to a globally
defined contact form on $M$, whose associated contact structure
will be denoted by $\xi$.

The contact manifold $(M,\xi)$ is determined by the data $(W,\lambda,\psi)$.
We shall call it a \textbf{contact open book}, and denote it by $\open( W,\psi^{-1})$.
Note that we use $\psi^{-1}$ rather than $\psi$ in this notation.
The following remark explains this.
\begin{remark}
A contact open book $\open(W,\psi^{-1})$ has the structure of a
fiber bundle over $S^1$ away from the set $B_W$. Hence we can talk
about the monodromy of an open book, which can be obtained by
lifting the tangent vector field to $S^1$, given by
$\partial_\phi$, to a vector field on $A_{(W,\psi)}$. If we
rescale the function $h$ to $2\pi$, then the time-$2\pi$ flow
gives the monodromy. Note that a positive function times the Reeb
field is a suitable lift of $\partial_\phi$. As a result, we see
that the monodromy is given by $\psi^{-1}$.
\end{remark}

\begin{definition}
  An \textbf{open book} on a manifold $M$ is a pair $(B,\Theta)$, where
\begin{itemize}
\item{} $B$ is a codimension $2$ submanifold of $M$ with trivial
  normal bundle, and
\item{} $\Theta:\, M-B\to S^1$ gives $M - B$ with the structure of a
  fiber bundle over $S^1$ such that $\Theta$ is equal to the angular
  coordinate of the $D^2$--factor on a neighborhood $B\times D^2$ of
  $B$.
\end{itemize}
\end{definition}

The set $B$ is called the \textbf{binding} of the open book. A fiber
of $\Theta$ together with the binding is called a \textbf{page} of the
open book.

Suppose $M$ is an oriented manifold with an open book
$(B,\Theta)$. 
We regard $S^1$ as an oriented manifold, so each page $W$ gets an induced orientation by requiring that the orientation of $M-B$, as a bundle over $S^1$, matches the one coming from $M$. 
If this induced orientation on $W$ coincides with its orientation as
a symplectic manifold $(W, \ow)$, then we call the symplectic form
$\ow$ \textbf{positive}. 
Now orient the binding $B$ as the boundary of a page $W$ using the outward normal first convention.
We say that $\alpha$ induces a {\bf positive contact structure} if this orientation of $B$ matches the one coming from a contact form $\alpha$.

\begin{definition}
  A positive contact structure $\xi$ on an oriented manifold $M$ is
  said to be \textbf{carried by an open book} $(B, \Theta)$ if $\xi$
  admits a defining contact form $\alpha$ satisfying the following
  conditions.
  \begin{itemize}
  \item{} $\alpha$ induces a positive contact structure on $B$, and
  \item{} $d\alpha$ induces a positive symplectic structure on each
    fiber of $\Theta$.
  \end{itemize}
  A contact form $\alpha$ satisfying these conditions is said to be
  \textbf{adapted} to $(B, \Theta)$.
If the above holds, the open book $(B,\Theta)$ is said to be a {\bf supporting open book} for $(M,\xi)$.
\end{definition}

Two well-known results are listed below.

\begin{lemma}
  Suppose that $B$ is a connected contact submanifold of a contact manifold $(M,\xi)$.  A
  contact form $\alpha$ for $(M,\xi)$ is adapted to an open book $(B,
  \Theta)$ if and only if the Reeb field $R_\alpha$ of $\alpha$ is
  positively transverse to the fibers of $\Theta$,
  i.e.~$R_\alpha(\Theta) > 0$.
\end{lemma}

\begin{proposition}
  A contact open book $\open(W, \psi^{-1})$ admits a natural
  open book carrying the contact structure $\xi$ as defined in the above
  construction.
\end{proposition}

\subsection{Fibered Dehn twists}
\label{sec:symplectomorphisms_monodromy} Suppose $(P,\theta)$ is a
contact manifold that admits an $S^1$-action generated by the flow
of the Reeb field $R_\theta$, i.e., Boothby--Wang orbibundles over
symplectic orbifolds. Now choose a function $f:[0,1] \to \R $ that
is constant $2\pi$ in a neighborhood of $0$ and constant $0$ in a
neighborhood of $1$. Then we can define a symplectomorphism of
$(P\times [0,1],d (e^t \theta ) )$ equal to the identity near the
boundary by sending
\[
\psi:(x,t) \longmapsto (x \cdot f(t)\, \mathrm{mod}\, 2\pi,t ).
\]
Since we also need to know the action of $\psi$ on $e^t \theta$
rather than just on $d(e^t \theta)$, let us compute $\psi^* (e^t
\theta)$. Observe that
\[
\begin{split}
\mathcal L_{f(t)R_\theta} (e^t \theta) &= d (\iota_{f(t)R_\theta} (e^t \theta))+\iota_{f(t)R_\theta} d(e^t \theta)\\
&=d(e^tf(t))-e^t f(t) dt\\
&=-d\left( A-e^tf(t)+\int^t_0 e^{s}f(s)ds \right),
\end{split}
\]
where $A$ is constant. Hence we have
$$
\psi^* (e^t \theta)=e^t \theta-d\left( A-e^tf(t)+\int^t_0
e^{s}f(s)ds \right),
$$
so we see in particular that $\psi$ is a symplectomorphism.

\begin{definition}
Let $(W,\omega)$ be a convex symplectic manifold whose boundary admits a quasi-regular contact form (i.e.~all Reeb orbits are periodic).
Define a symplectomorphism $\tilde \psi$ of $W$ by declaring $\tilde \psi$ to be equal to $\psi$ on a collar neighborhood of $\partial W$ and extending $\tilde \psi$ to be the identity on $W$ outside that neighborhood.
Such a symplectomorphism is called a {\bf right-handed fibered Dehn twist}.
\end{definition}

Consider a fibered Dehn twist on an exact convex symplectic manifold $(W,d\lambda)$.
Observe that the above computation allows us to avoid Lemma~\ref{lemma:deform_symplectomorphism}, since the symplectomorphism has already the appropriate form, i.e.~$
\psi^* \lambda=\lambda-dh$, where the function $h$ is only non-constant in a collar neighborhood of $\partial W$,
$$
h=A-e^tf(t)+\int^t_0 e^{s}f(s)ds.
$$

\subsection{Monodromy and fillability}
The monodromy of a contact open book provides information about the fillability of a contact manifold.
For our purposes, the following result is the most relevant.
\begin{theorem}
\label{thm:subcritical_open_book} A contact manifold $(M,\xi)$ is
subcritically Stein fillable if and only if there is a contact
open book $\open(W,\id)$ contactomorphic to $(M,\xi)$.
\end{theorem}
This assertion follows from a theorem of Cieliebak \cite{Cieliebak:Stein_split},
which asserts that subcritical Stein manifolds are split, and the simple observation that an open book with trivial monodromy can be written as
$$
\open(W,\id)=(\partial W\times D^2) \cup_\partial (W\times
S^1)=\partial (W\times D^2),
$$
where $W$ is Stein.
Suppose $f_W$ is a plurisubharmonic function on $W$.
Then we obtain a plurisubharmonic function on $W\times D^2$,
\begin{align*}
f: W\times D^2 & \longrightarrow  \R \\
(x,z) & \longrightarrow  f_W(x)+|z|^2,
\end{align*}
inducing the same contact structure as the one coming from an open
book with trivial monodromy.

\section{Maslov index}
\label{sec:maslov_index}

\subsection{Definition of a Maslov index using a crossing form}
Here we shall work with the Robbin-Salamon definition of the Maslov index, see \cite{Robbin:Maslovindex}.

Let $\omega_0$ denote the standard symplectic form on $\R^{2n}$ given by
$$
\omega_0=dx \w dy.
$$
\begin{definition}
Let $\psi:[0,T]\to Sp(2n)$ be a path of symplectic matrices.
We call a point $t\in [0,T]$ a {\bf crossing} if $\det (\psi(t)-\id)=0$.
For a crossing $t$, let $V_t=\ker (\psi(t)-\id)$ and define for $v\in V_t$ the quadratic form
$$
Q_t(v,v):=\omega_0(v,\dot \psi(t) v).
$$
The quadratic form $Q_t$ is called the {\bf crossing form} at $t$.
\end{definition}

Let us now define the Maslov index for symplectic paths in the following steps.
\begin{enumerate}
\item Take a path of symplectic matrices $\psi:[0,T]\to Sp(2n)$ and suppose that all crossings are isolated.
Suppose furthermore that all crossings are non-degenerate, i.e.~the crossing form $Q_t$ at the crossing $t$ is non-degenerate as a quadratic form.
\item Then we define the Maslov index for such paths $\psi$ as
$$
\mu(\psi)=\frac{1}{2} \sgn Q_0+\sum_{t\in (0,T) \text{ crossing} } \sgn Q_t +\frac{1}{2} \sgn Q_T
$$
Here $\sgn$ denotes the signature (i.e.~the number of positive
eigenvalues minus the number of negative eigenvalues) of a
quadratic form. For $*=0$ or $T$, $\sgn Q_* = 0$ if $*$ is not a
crossing.

According to Robbin and Salamon, $\mu(\psi)$ is invariant under
homotopies of the path $\psi$ with fixed endpoints. \item For a
general path of symplectic matrices $\psi:[0,T]\to Sp(2n)$, we
choose a perturbation $\tilde \psi$ of $\psi$ while fixing the
endpoints such that $\tilde\psi$ has only non-degenerate crossings.

\item Define
$$
\mu(\psi):=\mu(\tilde \psi).
$$
This is well defined according to Robbin and Salamon
\cite{Robbin:Maslovindex}.
\end{enumerate}

Recall that a non-degenerate Reeb orbit is a periodic Reeb orbit for which the restriction of the linearized return map to the contact structure has no eigenvalues equal to $1$.
To define the Conley-Zehnder index of a non-degenerate Reeb orbit $\gamma$ we choose a spanning disk $D_\gamma$ for $\gamma$ and trivialize the contact structure $\xi$ over $D_\gamma$. The linearized flow along $\gamma$ with respect to that trivialization then gives rise to a path of symplectic matrices, $\psi(t):=TFl^R_t(\gamma(0) \, )|_{\xi}$.
Then the {\bf Conley-Zehnder index} of $\gamma$ is given by
$$
\mu_{CZ}(\gamma):=\mu(\psi).
$$

\begin{remark}
\label{rem:CZ-index}
Some remarks are in order.
\begin{enumerate}
\item For a spanning disk to exist, the orbit $\gamma$ needs to be contractible. If we use Seifert surfaces rather than disks, then we can also consider homologically trivial periodic orbits.
\item The Conley-Zehnder index depends on the choice of spanning disk via the formula,
$$
\mu_{CZ}(\gamma;D'=D\# A)=\mu_{CZ}(\gamma,D)+2 \langle c_1(\xi),[A] \rangle.
$$
Here $A$ is a $2$-sphere, and $[A]$ its homology class.
To deal with this issue in symplectic homology (see Section~\ref{sec:s1eq}), one can consider coefficient rings other than $\Q$.
However, we shall only consider symplectic manifolds $W$ with contact type boundary for which $c_1(W)$ evaluates to $0$ on $\pi_2(W)$.
\item The Conley-Zehnder index is defined for non-degenerate orbits. We shall often use a Morse-Bott setup though.
In such a degenerate setup we shall say Maslov index (or Robbin-Salamon index): this notion is defined using the same scheme.
\item The Conley-Zehnder index appears in the index formula for the moduli space of Floer trajectories.
It plays the same role as the Morse index in Morse homology.
\item The Maslov index has useful properties, see \cite{Robbin:Maslovindex}.
The catenation property is of particular interest to us.
If $\psi_1$ and $\psi_2$ are symplectic paths with matching endpoints, then the Maslov index of the catenation is given by
$$
\mu(\psi_1 * \psi_2)=\mu(\psi_1)+\mu(\psi_2).
$$
It is important to use the Robbin-Salamon version of the Maslov index for this.
\end{enumerate}
\end{remark}

\section{$S^1$-equivariant symplectic homology}
\label{sec:s1eq}
In this section we briefly discuss $S^1$-equivariant symplectic homology, a symplectic deformation invariant of exact symplectic manifolds with contact type boundary.

Let $(W_0,\omega)$ be a compact exact symplectic manifold with
contact type boundary. Symplectic homology is a Floer-like
homology theory that measures both information about periodic Reeb
orbits on the boundary and information about the filling. The
version of symplectic homology that we shall use, is
$S^1$-equivariant symplectic homology. $S^1$-equivariant homology
is a special version of parametrized symplectic homology,
introduced first by Viterbo \cite{Viterbo:Floer_applications}, and
worked out by Bourgeois and Oancea
\cite{Bourgeois:Gysin_S1_equivariant_symplectic_homology}.

We survey Bourgeois and Oancea's approach.
The idea is to think of $S^1$ as acting on $S^{2N+1}$ and to take the limit $N\to\infty$ in order to have a model for $ES^1$.
Using such a model one can apply the Borel construction to symplectic homology.

First of all, we complete the symplectic manifold $W_0$ by attaching the positive part of a symplectization: define
$$
W:=W_0 \bigcup_\partial (\R_{\geq 0}\times \partial W_0),
$$
where the symplectic form on the symplectization part is given by $d(e^t \alpha)$.
Here $\alpha$ is obtained from the Liouville form for $\omega$ by restricting to $\partial W_0$.

The action spectrum $\Spec (\alpha)$ of $(\partial W_0,\alpha)$ is
defined by
$$\Spec(\alpha):=\bigl\{ T\in \R^+ \mid {\rm there\ is \ a \
closed\  Reeb\  orbit \ of \ period}\  T\bigr\}.$$

Choose a Hamiltonian family $H\in C^{\infty} ( S^1\times W\times
S^{2N+1}, \R )$ with the following properties:
\begin{enumerate}
\item $H<0$ on $S^1\times W_0\times S^{2N+1}$; \item there exists
$t_0\ge 0$ such that $H(\theta,t,q,\lambda)=se^t+\beta (\lambda)$
for $t\ge t_0$ with $0<s\notin \Spec (\alpha)$ and $\beta\in
C^{\infty}(S^{2N+1},\R )$.
\end{enumerate}
Such a Hamiltonian family is called an admissible Hamiltonian
family. Consider the family of action functionals
\begin{align*}
\mathcal A: C^{\infty}_{\rm contr}(S^1,W)\times S^{2N+1} & \longrightarrow \R \\
(\gamma,\lambda) & \longmapsto {\mathcal
A}_\lambda(\gamma)=-\int_{D^2} \sigma ^*
\omega-\int_{S^1}H_\lambda(\theta,\gamma(\theta)\, )d\theta.
\end{align*}
Here $C^{\infty}_{\rm contr}(S^1,W)$ denotes the space of smooth
contractible loops in $W$, $\sigma:D^2 \to W$ is a spanning disk
for $\gamma$ and $H_{\lambda} (\theta,x)=H(\theta,x,\lambda)$.

\begin{remark}
This can be generalized to non-contractible loops $\gamma$ by
choosing reference loops for all free homotopy classes of loops in
$W$ and taking for $\sigma:[0,1]\times S^1\to W$ a homotopy from
such a reference loop to $\gamma$. Note also that if
$(\gamma,\lambda)$ is a critical point of this action functional,
then $\gamma$ is a $1$-periodic \emph{Hamiltonian orbit} of
$X_{H_{\lambda}}^{\theta}$ (in
\cite{Bourgeois:Gysin_S1_equivariant_symplectic_homology},
$S^1=\R/\Z$). Here we use the convention
$$
i_{X_{H_{\lambda}}^{\theta}}\omega=dH_{\lambda,\theta},
$$
where $H_{\lambda,\theta}(x)=H(\theta,x,\lambda)$, to define the
Hamiltonian vector field.
\end{remark}

Let $\{ J=(J^{\theta}_{\lambda}),\ \lambda\in S^{2N+1},\ \theta\in
S^1 \}$ be a family of $\theta$-dependent compatible almost complex
structures on $W$ which, at infinity, are invariant under
translations in the $\R$-direction and satisfy the relations
$J^{\theta}_{\lambda}\xi=\xi,\
J^{\theta}_{\lambda}(\partial_t)=R_{\alpha}$, where $\xi$ is the
contact structure $\ker\alpha$ on $\partial W_0$.
\begin{definition}
Such an admissible family of almost complex structures
$J=(J^\theta_\lambda)$ is called $S^1$-invariant if
$$
J^{\theta+\tau}_{\tau \lambda}=J^\theta_\lambda.
$$
\end{definition}
Note that here $S^1$ acts diagonally on $S^1\times S^{2N+1}$.
Given such a family of almost complex structures we obtain a
family of $L^2$-metrics on $C^{\infty}(S^1,W)$, parametrized by
$S^{2N+1}$: for $X,Y\in
T_{\gamma}C^{\infty}(S^1,W)=\Gamma(\gamma^*TW)$,
$$
\langle X,Y \rangle_{\lambda}=\int_{S^1} \omega(X(\theta),J^\theta_\lambda Y(\theta)\,)d\theta.
$$
Finally, we also need an $S^1$-invariant metric $g$ on the
parameter space $S^{2N+1}$ to write down the flow equations.

Let $H: S^1\times W\times S^{2N+1}\to\R$ be an admissible
Hamiltonian family which is $S^1$-invariant, i.e.
$H(\theta+\tau,\cdot,\tau \lambda)=H(\theta,\cdot,\lambda)$. We
denote by $\mathcal{P}^0(H)$ the set of critical points of
$\mathcal{A}$. Since $H$ is $S^1$-invariant, the family
$\mathcal{A}$ is invariant with respect to the diagonal action of
$S^1$, i.e. $\mathcal{A}(\tau\gamma,\tau
\lambda)=\mathcal{A}(\gamma,\lambda)$, where
$(\tau\gamma)(\cdot)=\gamma(\cdot-\tau),\ \tau\in S^1$. Thus
$\mathcal{P}^0(H)$ is $S^1$-invariant, i.e. if
$(\gamma,\lambda)\in \mathcal{P}^0(H)$, then
$(\tau\gamma,\tau\lambda)\in \mathcal{P}^0(H)$ for all $\tau\in
S^1$. Given $p=(\gamma,\lambda)\in\mathcal{P}^0(H)$, we denote
$S_p=S_{(\gamma,\lambda)}:=\{ (\tau\gamma,\tau\lambda):\tau\in
S^1\}\subset \mathcal{P}^0(H)$.

The ``gradient'' flow of the action gives rise to the parametrized Floer equation
for a pair $(u,\lambda)$, where $u:\R\times S^1\to W$ and $\lambda:\R \to S^{2N+1}$.
These equations and initial conditions are as follows.
\begin{align*}
\partial_s u+J^\theta_{\lambda(s)} \partial_\theta u-J^\theta_{\lambda(s)} X^\theta_{H_{\lambda(s)}}(u)&=0, \\
\dot \lambda(s)-\int_{S^1}\vec \nabla_{\lambda}H(\theta,u(s,\theta),\lambda(s))d\theta &=0,\\
\lim_{s\to -\infty} (u(s,\cdot),\lambda(s)\, ) &\in S_{\bar p}, \\
\lim_{s\to \infty} (u(s,\cdot),\lambda(s)\, ) &\in S_{\underline
p},
\end{align*}
where $\bar{p},\underline{p}\in \mathcal{P}^0(H)$. Denote by
$\mathcal M(S_{\bar p},S_{\underline p};H,J,g)$ the moduli space
of such Floer trajectories up to reparametrization. The
reparametrization action is here given by the $\R$-action on the
$s$-coordinate.

\subsection{Chain complex and differential}
Let $H: S^1\times W\times S^{2N+1}\to\R$ be an admissible
Hamiltonian family which is $S^1$-invariant and satisfies the
following: every $S^1$-orbit of critical points $S_p\subset
\mathcal{P}^0(H)$ is non-degenerate in the sense that the Hessian
$d^2\mathcal A(\gamma,\lambda)$ has a $1$-dimensional kernel for
some (and hence any) $(\gamma,\lambda)\in S_p$. We denote by
$\mathcal{H}_{N,{\rm reg}}^{S^1}$ the set of such Hamiltonian
families.

Define the $S^1$-equivariant chain complex $SC_*^{S^1,N}(H,J,g)$
as a chain complex whose underlying $\Q$-vector space is
$$
SC_*^{S^1,N}(H,J,g):=\bigoplus_{S_p \subset \mathcal P^0(H)} \Q
\langle S_p \rangle .
$$
The grading of each $S_p$ is given by
$$
|S_p|=-\mu(p)+N+\frac{1}{2}.
$$
See \cite{Bourgeois:Gysin_S1_equivariant_symplectic_homology} for
the definition of $\mu(p)$. We assume here that $c_1(W_0)$ is a
torsion class, see Remark~\ref{rem:CZ-index}.
\begin{remark}
Note that there is a sign and a shift with respect to the conventions of contact homology: the sign is necessary
since we are considering Hamiltonian orbits, where we use the convention that $i_{X_H}\omega=dH$.
With this definition, the Hamiltonian vector field runs in the direction opposite to that of the Reeb field.
\end{remark}

Since $\mathcal{A}$ and $(J,g)$ are $S^1$-invariant, $\mathcal M
(S_{\bar p},S_{\underline p};H,J,g)$ carries a free action of
$S^1$ induced by the diagonal action on $C^{\infty}(S^1,W)\times
S^{2N+1}$, i.e. $\tau
(u(\cdot),\lambda):=(u(\cdot-\tau),\tau\lambda)$. We denote the
quotient by $\mathcal M_{S^1}(S_{\bar p},S_{\underline
p};H,J,g):=\mathcal M (S_{\bar p},S_{\underline p};H,J,g)/S^1$.
According to Bourgeois and Oancea
\cite{Bourgeois:Gysin_S1_equivariant_symplectic_homology}, this is
a smooth manifold of dimension
$$
\dim \mathcal M_{S^1}(S_{\bar p},S_{\underline p};H,J,g)=-\mu( \bar p)+\mu( \underline p ) -1
$$
if we choose a suitable perturbation data $(J,g)$ for the
Hamiltonian family $H$. Hence the following definition for the
differential makes sense,
$$
\partial^{S^1} (S_{\bar p})=\sum_{\substack{S_{\underline p}\subset\mathcal{P}^0(H)\\ -\mu( \bar p)+\mu( \underline p )= 1
}}
 \left( \# \mathcal M_{S^1}(S_{\bar p},S_{\underline p};H,J,g) \right ) S_{\underline p},
$$ where $\#$ is a signed count of the number of elements of
$\mathcal M_{S^1}(S_{\bar p},S_{\underline p};H,J,g)$. One can
prove that $\partial^{S^1}$ is actually a differential, so
$\partial^{S^1} \circ \partial^{S^1}=0$.

Next, define the $S^1$-equivariant Floer homology groups by taking
the homology,
$$
SH_*^{S^1,N}(H,J,g):=H_*(SC_*^{S^1,N}(H,J,g),\partial^{S^1}).
$$
One can show that these Floer homology groups do not depend on the
choice of perturbation data $(J,g)$, so we shall write
$SH_*^{S^1,N}(H)$ from now on. Taking the direct limit over the
Hamiltonians as the non-equivariant symplectic homology:
$$
SH_*^{S^1,N}(W_0,\omega):=\varinjlim_{H\in \mathcal{H}_{N,{\rm
reg}}^{S^1}} SH_*^{S^1,N}(H).
$$
To complete the construction take the direct limit over $N$,
$$
SH_*^{S^1}(W_0,\omega):=\varinjlim_{N} SH_*^{S^1,N}(W_0,\omega).
$$

\subsection{Subcomplexes and relation to symplectic homology}
Instead of taking the direct limits, we can also first investigate
subcomplexes. For sufficiently small $\epsilon >0$ define the
subcomplex
$$
SC_*^{S^1,-,N}(H,J,g):= \bigoplus_{\substack{S_p \subset \mathcal
P^0(H) \\ \mathcal A(p)\leq \epsilon}} \Q \langle S_p \rangle .
$$
This leads to the quotient complex
$$
SC_*^{S^1,+,N}(H,J,g):=SC_*^{S^1,N}(H,J,g)/SC_*^{S^1,-,N}(H,J,g).
$$
For either of these groups we can define direct limits over $H$
and $N$ as in the two steps for $SH_*^{S^1}(W_0,\omega)$, leading to the
homology groups $SH_*^{S^1,\pm}(W_0,\omega)$.

Equivariant symplectic homology is related to non-equivariant
symplectic homology through a Gysin sequence. We have
$$
\ldots \longrightarrow SH^b_*(W_0,\omega) \longrightarrow
SH^{S^1,b}_* (W_0,\omega)  \longrightarrow SH^{S^1,b}_{*-2}
(W_0,\omega) \longrightarrow SH^b_{*-1}(W_0,\omega)
\longrightarrow \ldots
$$
Here $b$ is used to denote any of the three types of complexes: the full complex, the $-$ complex and the $+$ complex.

\subsection{Homological boundedness, index positivity and exact sequences}

Assume that $(W,\omega=d\lambda)$ is a compact exact symplectic
manifold, i.e. $\omega=d\lambda$ is a symplectic form on $W$, with
convex boundary $\partial W$. We assume that the first Chern class
$c_1(W)$ of $(W,\omega)$ is a torsion class.

\begin{definition}
We say $(W,\omega)$ is {\bf homologically bounded} if there exists
$C>0$ such that $b_i(W,\omega)=\dim (SH_i^{S^1,+}(W,\omega))<C$
for all $i\in\Z$.
\end{definition}

\begin{definition}
We say that a cooriented contact manifold $(\Sigma,\alpha)$ is
{\bf index-positive} if the mean index $\Delta(\gamma)$ of every contractible, periodic Reeb orbit~$\gamma$ is positive.
Similarly, we say that $(\Sigma,\alpha)$ is
{\bf index-negative} if the mean index $\Delta(\gamma)$ of every contractible, periodic Reeb orbit $\gamma$ is negative.
Finally, we say that $(\Sigma,\alpha)$ is {\bf index-definite}
if it is index-positive or index-negative.
\end{definition}
Recall that
the {\bf mean index} $\Delta$ is related to the Conley--Zehnder index $\mu_{CZ}$ as follows:
For any non-degenerate Reeb orbit $\gamma$ in a contact manifold $(\Sigma^{2n-1},\alpha)$, its $N$-fold cover $\gamma^N$ satisfies
\begin{equation}
\label{eq:iteration_and_mean_index}
\mu_{CZ}(\gamma^N) \,=\, N \Delta(\gamma) + e(N),
\end{equation}
where $e(N)$ is an error term bounded by~$n-1$, see~\cite[Lemma
3.4]{SalamonZehnder:Morse}.

There is also a homological version of this notion. 
\begin{definition}
We shall call
a homology $H_*(C_*,\partial)$ {\bf index-positive} if there
exists $N$ such that $H_i(C_*,\partial)=0$ for all $i<N$.
\end{definition}
If $(\Sigma,\alpha=i_{X_{Liouville}} d\lambda)=\partial(W,d\lambda)$ is
index-positive in the previously defined sense, and if the inclusion of $\Sigma$ into $W$ induces an injection on $\pi_1$, then
$SH^{S^1,+}_*(W,d\lambda)$ is index-positive in the homological
sense.
One way to see this, is to use a spectral sequence argument similar to the proof of Proposition~\ref{prop:mean_euler_S^1-orbibundle}.
The notions index-negative and index-definite are defined
on homology level in a similar way.

We observe that for a compact subcritical Stein manifold
$(W^{2n},\omega)$ with torsion first Chern class,
$SH_*^{S^1,+}(W,\omega)$ is index-positive. To see this, we consider
Corollary~1.3 from
\cite{Bourgeois:Gysin_S1_equivariant_symplectic_homology} which
states that there is an isomorphism of exact sequences,
\[
\entrymodifiers={+!!<0pt,\fontdimen10\textfont2>}
\xymatrix{
\cdots \ar[r] & SH^{+}_*(W,\omega) \ar[r] \ar[d]^{\cong}&
SH^{S^1,+}_*(W,\omega) \ar[r] \ar[d]^{\cong} &
SH^{S^1,+}_{*-2}(W,\omega) \ar[r] \ar[d]^{\cong} &
SH^{+}_{*-1}(W,\omega) \ar[r] \ar[d]^{\cong} & \cdots \\
\cdots \ar[r] & H_{*+n-1}(W,\partial W) \ar[r] &
H^{S^1}_{*+n-1}(W,\partial W) \ar[r] & H^{S^1}_{*+n-3}(W,\partial
W) \ar[r] &
H_{*+n-2}(W,\partial W) \ar[r] & \cdots \\
}
\]
As the equivariant homology of $(W,\partial W)$ is index-positive, we see that $SH_*^{S^1,+}(W,\omega)$ is index-positive.

\subsection{Euler characteristic and mean Euler characteristic}
To simplify computations we shall use the mean Euler
characteristic of the positive part of the $S^1$-equivariant
symplectic homology. This number can be computed explicitly for
certain classes of symplectic manifolds. Furthermore, it can be
used to detect the non-existence of subcritical fillings, see Proposition~\ref{proposition:mean_euler_subcritical}, and also
serves as an obstruction against the existence of displaceable
exact contact embeddings, see
\cite{Frauenfelder:mean_euler_characteristic_displaceability}.

Let $(W,\omega)$ be a compact exact symplectic manifold with
contact type boundary. Suppose that $(W,\omega)$ is homologically
bounded. We define the {\bf mean Euler characteristic} of
$(W,\omega)$ as
$$
\chi_m(W,\omega)=\frac{1}{2}\left(\liminf_{N\to \infty}
\frac{1}{N}\sum_{i=-N}^N(-1)^i b_{i}(W,\omega)+\limsup_{N\to
\infty}\frac{1}{N}\sum_{i=-N}^N(-1)^i b_{i}(W,\omega)\right).
$$
The uniform bound on $b_i(W,\omega)$ implies that the limit
inferior and the limit superior exist. See also
\cite{vanKoert:thesis}, \cite{Ginzburg:homological_resonances} and
\cite{Espina:mean_euler_characteristic}. In some cases, the mean
Euler characteristic is independent of the filling, that it can be computed in terms of data on $\partial W$ only, see \cite{Frauenfelder:mean_euler_characteristic_displaceability}, which allows us to use
this number as a \emph{contact invariant}. For later applications,
the main observation is that the mean Euler characteristic of
compact subcritical Stein manifolds is always a half-integer,
\begin{proposition}
\label{proposition:mean_euler_subcritical} 
Let $(Y^{2n-1},\xi=\ker
\alpha)$ be a contact manifold with subcritical filling
$(W^{2n},\omega)$ such that $c_1(W)$ is a torsion class. Then
$SH_*^{S^1,+}(W,\omega)\cong H_{*+n-1}(W,\partial W;\Q) \otimes H_*( \C P^\infty;\Q)$.
In particular, it is index-positive with generators in arbitrarily large degrees.
Furthermore,
$$
\chi_{m}(W,\omega)=\frac{(-1)^{n+1} \chi(W)}{2}.
$$
\end{proposition}

\begin{lemma}
\label{lemma:meanGysin}
Suppose we have a Gysin style exact sequence for $H_*(B)$ and $H_*^{S^1}(B)$ of the form
\[
\entrymodifiers={+!!<0pt,\fontdimen10\textfont2>}
\xymatrix{
\cdots \ar[r] &
H_*(B) \ar[r] &
H^{S^1}_*(B) \ar[r]  &
H^{S^1}_{*-2}(B) \ar[r] &
H_{*-1}(B) \ar[r] & \cdots \\
}
\]
Suppose furthermore that $H_i(B)$ is finite dimensional for all
$i\in\Z$, there exists a positive integer $N_0$ such that
$H_i(B)=0$ for all $i>N_0$ and all $i<-N_0$, $\dim (H_i^{S^1}(B))$
are uniformly bounded, i.e. there exists $C>0$ such that $\dim
(H_i^{S^1}(B))<C$ for all $i\in\Z$, and $H_*^{S^1}(B)$ is
index-definite. Then
$$
\chi_m(H_*^{S^1}(B)\, )=\pm \frac{\chi(H_*(B)\, )}{2},
$$
where one should take a $+$ sign if $H_*^{S^1}(B)$ is
index-positive, and a $-$ sign if $H_*^{S^1}(B)$ is
index-negative. The definition of $\chi_m(H_*^{S^1}(B)\, )$ is
similar to that of $\chi_{m}(W,\omega)$ (replacing $b_i(W,\omega)$
by $\dim (H_i^{S^1}(B))$ in that definition).
\end{lemma}

\begin{proof}
First of all, observe that the conditions on $H_*(B)$ guarantee
that $\chi(H_*(B)\,)$ exists. Secondly, for $N>N_0$ we have
\[
\begin{split}
0&=\sum_{i=-N}^{N}(-1)^i \dim (H_i(B)) +\sum_{i=-N}^{N}(-1)^{i+1} \dim (H_i^{S^1}(B)) +\sum_{i=-N}^{N}(-1)^i \dim (H_{i-2}^{S^1}(B))\\
&=\chi( H_*(B)\,)+(-1)^{N} \dim (H_{N-1}^{S^1}(B)) +(-1)^{N+1}
\dim
(H_{N}^{S^1}(B))\\
&+(-1)^{-N}\dim (H_{-N-2}^{S^1}(B))+(-1)^{-N+1}\dim
(H_{-N-1}^{S^1}(B)).
\end{split}
\]

Since we assume that $H^{S^1}_*(B)$ is index-definite, we have two cases to consider.
\begin{enumerate}
\item $H^{S^1}_*(B)$ is index-positive.
Then for sufficiently large $N$ we have
$$
\chi(H_*(B)\,)=(-1)^{N-1} \dim (H_{N-1}^{S^1}(B))+(-1)^{N} \dim
(H_{N}^{S^1}(B)).
$$
Hence
\[
\begin{split}
\chi_m(H^{S^1}_*(B)\,)&=\lim_{N \to \infty} \frac{1}{N} \sum_{i=-N}^{N}(-1)^i \dim (H_i^{S^1}(B))\\
&= \lim_{N \to \infty} \frac{1}{N} \sum_{i=N_0}^{N}(-1)^i \dim
(H_i^{S^1}(B)) =\frac{\chi(H_*(B)\, )}{2}.
\end{split}
\]
\item The proof in the index-negative case is very similar, but
there is a sign change since now for sufficiently large $N$ we
have
$$
\chi(H_*(B)\,)=(-1)^{-N-1} \dim (H_{-N-2}^{S^1}(B))+(-1)^{-N} \dim
(H_{-N-1}^{S^1}(B)).
$$
\end{enumerate}

\end{proof}

\begin{proof}[Proof of Proposition~\ref{proposition:mean_euler_subcritical}]
Observe that $SH_*(W,\omega)$ vanishes by a result of Cieliebak
\cite{Cieliebak:vanishingSH}. Hence the Viterbo long exact
sequence, see \cite{Viterbo:Floer_applications}, reduces to
$$
0\cong SH_*(W,\omega) \longrightarrow SH^+_{*}(W,\omega)
\longrightarrow H_{*+n-1}(W,\partial W) \longrightarrow
SH_{*-1}(W,\omega)\cong 0.
$$
Furthermore, if $c_1(W)$ is a torsion class, then it follows from
the preceding section that $SH_*^{S^1,+}(W,\omega)$ is
index-positive. Hence Lemma~\ref{lemma:meanGysin} applies and we obtain
\[
\begin{split}
\chi_m(W,\omega)& =\chi_m(SH_*^{S^1,+}(W,\omega)\,)=\frac{
\chi(SH_*^{+}(W,\omega)\,)}{2}=(-1)^{n-1}\frac{\chi ( H_{*+n-1}(W,\partial
W)\,)}{2}\\
&=(-1)^{n-1}\frac{\chi(W)}{2}.
\end{split}
\]
\end{proof}

\begin{remark}
This proposition can also be proved by using Yau's results on the
contact homology of subcritically fillable contact manifolds,
see~\cite{Yau:subcritical}. Alternatively, one can use Espina's
argument, see \cite[Corollary
5.7]{Espina:mean_euler_characteristic}, which tells us that
subcritical surgery changes the mean Euler characteristic by $\pm
\frac{1}{2}$. Since the mean Euler characteristic of
$(S^{2n-1},\xi_0)=\partial (D^{2n},\omega_0)$ is
$\frac{(-1)^{n+1}}{2}$, the result follows by successive handle
attachments.

Also observe that the above result holds true for a set bounded by a displaceable contact embedding.
In such a case, one can apply \cite[Theorem 97]{Ritter:topological_quantum_field_SH}.
See also \cite{Frauenfelder:mean_euler_characteristic_displaceability}.
\end{remark}

\begin{remark}
Note that grading conventions in contact homology differ from the ones in symplectic homology.
As a result we have the sign $(-1)^{n+1}$.
\end{remark}

We shall now consider the case of Boothby--Wang orbibundles.
By introducing the notion of Morse-Bott contact form we can avoid using perturbations of the contact form.
\begin{definition}
A contact form $\alpha$ on $\Sigma$ is said to be of {\bf
Morse--Bott type} if the following hold:
\begin{itemize}
\item The action spectrum $\Spec  (\alpha)$ is discrete. \item For
every $T\in \Spec (\alpha)$, $N_{T}=\{ p\in \Sigma|
Fl^{R_\alpha}_{T}(p)=p\}$ is a smooth submanifold of $\Sigma$ such
that the rank $d\alpha|_{N_{T}}$ is locally constant and
$T_pN_{T}=\ker (TFl^{R_\alpha}_{T}-id)_p$.
\end{itemize}
\end{definition}
To avoid orientation problems of the moduli spaces, we need the notion of bad orbit in Morse-Bott sense. Let $N_T$ denote the submanifold consisting of periodic orbits with period $T$, that is
$$
N_T=\{ p \in \Sigma ~|~Fl^{R_\alpha}_T(p)=p \}
$$
For a periodic Reeb orbit $\gamma\subset N_T$, the Maslov index is independent of the choice of $\gamma$. Hence we write $\mu(N_T)$ rather than $\mu(\gamma)$.

\begin{definition}
\label{def:bad_orbit}
A Reeb orbit $\gamma$ of period $T$ is called {\bf bad} if it is the $2m$ fold cover of a simple orbit $\gamma'$ and $\mu(N_T)-\frac{1}{2 }\dim \left( N_T/S^1 \right) -\mu( N_{\frac{T}{2m}})+\frac{1}{2}\dim \left( N_{\frac{T}{2m}}/S^1 \right)$ is odd.
\end{definition}

We introduce some notation to state the result.
Consider a contact manifold $(\Sigma,\alpha)$ with Morse--Bott
contact form $\alpha$ such that all Reeb orbits are periodic, so that
we have an $S^1$-action on $\Sigma$ (not necessarily free). Denote
the minimal periods by $T_1< \ldots< T_k$, so all $T_i$ divide~$T_k$.
As before, denote the subspace consisting of periodic Reeb orbits
with period $T_i$ in $\Sigma$ by~$N_{T_i}$. For the proof of the
following lemma, see
\cite{Frauenfelder:mean_euler_characteristic_displaceability}.

\begin{lemma}
\label{lemma:H^1_trivial} If $H^1(N_{T_i};\Z_2)=0$, then
$H^1(N_{T_i} \times_{S^1} ES^1; \Z_2)=0$.
\end{lemma}

The proof of the following proposition is similar to that of
\cite[proposition 2.4]{Frauenfelder:mean_euler_characteristic_displaceability}. The main difference is that we need to keep track of the homotopy class of periodic orbits.
For later use the following definition is useful. 
Let $(P^{2n-1},\alpha)$ be a cooriented contact manifold whose Reeb flow is periodic. Assume furthermore that the following holds.
\begin{enumerate}
\item[P1] there is a connected set of exceptional orbits $N_{T_1}$ with period $T_1$. The principal orbits, denoted by $N_{T_2}$, have period $T_2$.
\item[P2] $\pi_1(P)\cong \Z_k$, generated by a simple exceptional orbit. Furthermore, $c_1(\xi)$ is torsion.
\item[P3] The Maslov index of the smallest contractible cover of a principal orbit, denoted by $\mu_P$, is non-zero. 
\end{enumerate}
Write $N:=T_2/T_1$, and put $\ell=\gcd(N,k)$.
Define the {\bf mean Euler characteristic} of $(P,\alpha)$ by the number
\begin{equation}
\label{eq:definition_mean_euler_contact}
\chi_m(P,\alpha ) =(-1)^{n+1}
\frac{(\frac{N}{\ell}-1)\chi^{S^1}(N_{T_1})+\chi^{S^1}(N_{T_2})}{|\mu_P|}.
\end{equation}
Here $\chi^{S^1}(N_T)$ denotes the Euler characteristic of the
$S^1$-equivariant homology of the $S^1$-manifold~$N_T$.
In general, this is a meaningless number, but the following proposition shows that it is a contact invariant provided that there is a suitable filling.
\begin{proposition}
\label{prop:mean_euler_S^1-orbibundle}
Let $(P^{2n-1},\xi=\ker \alpha)$ be a cooriented contact manifold
satisfying the following conditions:
\begin{enumerate}
\item[P1-3] The conditions P1, P2 and P3 hold. In particular $\mu(N_{\frac{Nk}{\ell}T_1}=N_{\frac{k}{\ell}T_2})=\mu_P$. 
\item[P4] Furthermore  $H^1(N_T\times_{S^1} ES^1; \Z_2)=0$ for all
$N_T$ and there are no bad orbits.
\item[P5] There is an exact
filling $(W^{2n},d\lambda)$ such that the inclusion $P \to W$ induces an injection on $\pi_1$. 
In addition, $c_1(W)$ is torsion.
\end{enumerate}
Then the mean Euler characteristic of equivariant symplectic
homology in the class of contractible orbits is given by
$$
\chi_m( SH^{S^1,+}_*(W,d\lambda )\, ) =\chi_m(P,\alpha).
$$
\end{proposition}

\begin{remark}
This proposition is a generalization of \cite[Example
8.2]{Espina:mean_euler_characteristic}, and Espina's methods could
also be used to show the above.
\end{remark}

\begin{proof}
The Reeb flow on $P$ is periodic, so we can use Morse--Bott methods to construct a spectral sequence converging to $SH^{S^1,+}_*(W,d\lambda)$, see \cite[Section 7.2.2]{FOOO}.
Its $E^1$-page is given by
$$
E^1_{pq}=\bigoplus_{\substack{N_T \text{ consists of contractible orbits} \\ \mu(N_T)-\frac{1}{2}\dim (N_T/S^1)=p}} H^{S^1}_{q}(N_T;\Q).
$$
See also Seidel, \cite{Seidel:SH} formula 3.2, for a similar spectral sequence for symplectic cohomology with different conventions.
Note that the sum is over all orbit spaces of contractible orbits including multiple covers.
Since we have two orbit types, namely corresponding to $N_{T_1}$ and $N_{T_2}$, we can split the direct sums as
$$
E^1_{*q}=
\bigoplus_{m>0 \text{ such that }km\notin N \Z}H^{S^1}_{q}(N_{kmT_1};\Q) \oplus  \bigoplus_{m'>0}H^{S^1}_{q}(N_{\frac{m'Nk}{\ell}T_1};\Q).
$$
Indeed, if $km\in N \Z$ for the first term, then the orbits are part of the larger orbit space $N_{\frac{m'Nk}{\ell}T_1}=N_{\frac{m'k}{\ell}T_2}$ with $m'\frac{N}{\ell}=m$, which we count in the second term. The second term consists of contractible covers of principal orbits.
We have indicated what happens pictorially in Figure~\ref{fig:MB_spectral_sequence}.
\begin{figure}[htp]
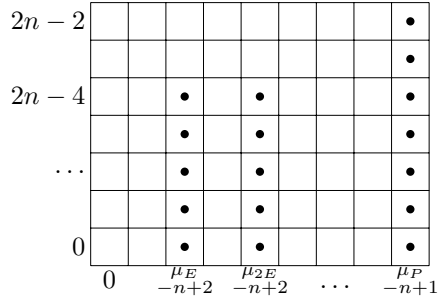

\begin{center}
\begin{sseq}[entrysize=5mm,xlabelstep=2,ylabelstep=2,
xlabels={0;?;\substack{\mu_E\\-n+2};\ldots;\substack{\mu_{2E}\\-n+2};\ldots;\ldots;\substack{\mu_P\\-n+1}},ylabels={0;2;\ldots;?;2n-4;?;2n-2}]
{0...8}{0...6}
\ssmoveto 2 0
\ssdropbull
\ssmove 0 1
\ssdropbull
\ssmove 0 1
\ssdropbull
\ssmove 0 1
\ssdropbull
\ssmove 0 1
\ssdropbull

\ssmoveto 4 0
\ssdropbull
\ssmove 0 1
\ssdropbull
\ssmove 0 1
\ssdropbull
\ssmove 0 1
\ssdropbull
\ssmove 0 1
\ssdropbull

\ssmoveto 8 0
\ssdropbull
\ssmove 0 1
\ssdropbull
\ssmove 0 1
\ssdropbull
\ssmove 0 1
\ssdropbull
\ssmove 0 1
\ssdropbull
\ssmove 0 1
\ssdropbull
\ssmove 0 1
\ssdropbull

\end{sseq}

\end{center}
\caption{A single period in the $E^1$-page of Morse-Bott spectral sequence for $SH^{S^1,+}_*(W,d\lambda)$}
\label{fig:MB_spectral_sequence}
\end{figure}
Since the flow is periodic, the spectral sequence repeats itself after reaching the block consisting of contractible covers of principal orbits.
Hence we count the contribution of each block, which either corresponds to $H^{S^1}_{q}(N_{kmT_1};\Q)$ or to $H^{S^1}_{q}(N_{\frac{mk}{\ell}T_2};\Q)$, to the Euler characteristic.
We see that $\frac{N}{\ell}-1$ copies of $H^{S^1}_{q}(N_{kmT_1};\Q)$ occur before a block of principal orbits appears.
The blocks repeat with degree shift of $\mu_P$, since the flow is periodic.

To determine the signs of each contribution, we observe that $\mu_P$ is even.
To see this, note that $\mu_P=\mu(N_{\frac{k}{\ell}T_2})$.
Since the flow is periodic for the principal orbits, we have $\mu(N_{\frac{mk}{\ell}T_2})=m \mu(N_{\frac{k}{\ell}T_2})$.
Since all orbits are assumed to be good, this can only hold if $\mu(N_{\frac{k}{\ell}T_2})=\mu_P$ is even.
The contribution of the principal orbits to the mean Euler characteristic is hence $(-1)^{\mu_P-\dim \left( N_{T_2}/S^1 \right)}\chi^{S^1}(N_{T_2})=(-1)^{n-1}\chi^{S^1}(N_{T_2})$.

Finally, we claim that the Maslov indices of the contractible covers of the exceptional orbits that are not contained in a space of principal orbits, are odd. We write $\mu_E,\mu_{2E},\ldots$ for these Maslov indices. 
Indeed, suppose that $\mu_E$ is even.
Note that an $N$-fold cover of an exceptional orbit is principal, so by Definition~\ref{def:bad_orbit} we have that $\mu_{NE}-n+1-\mu_E+n-2$ is odd. This contradicts the non-existence of bad orbits.
We note that the exceptional orbits contribute $(\frac{N}{\ell}-1)(-1)^{\mu_E-\dim \left( N_{T_1}/S^1 \right)}\chi^{S^1}(N_{T_1})=(-1)^{n-1}(\frac{N}{\ell}-1)\chi^{S^1}(N_{T_1})$.
\end{proof}

\subsection{Maslov index for simple Boothby-Wang bundles}
In this section we consider Boothby-Wang bundles for which multiple covers of the $S^1$ fibers are contractible.

Let $(M,\omega)$ be a compact simply connected symplectic manifold
such that $[\omega]\in H^2(M;\Z)$ is primitive. Consider a
Boothby--Wang bundle $P_M$ associated with the symplectic manifold
$(M, k\omega)$, and denote the projection $P_M\to M$ by $\Pi$.

Suppose that $c_1(M)=c [\omega]$ for some $c\in\Z$. This implies
that $c_1(\xi)=-\Pi^* (c_1(M))$ is a torsion cohomology class, so
we can use $\Q$ as a coefficient ring of symplectic homology. See
Chapter~9 from \cite{Bourgeois:thesis} for the following lemma.
\begin{lemma}
\label{lemma:degree_shift_BW}
The Maslov index of a $k$-fold cover of a principal orbit $S$ is given by $\mu(S)=2c$.
\end{lemma}

\begin{remark}
\label{remark:multiple_cover_contractible} In this setup,
$\pi_1(P_M)\cong\Z_k$, so a $k$-fold cover of a principal orbit is
contractible. Furthermore, if $\gamma$ is a principal orbit, then
$[\gamma]\in \pi_1(P_M)$ represents $1\in \Z_k$.
\end{remark}

\section{Open books for Boothby--Wang bundles}
\label{sec:bw}

In this section, we discuss the topological conditions for a
Boothby--Wang bundle to possess a specific supporting open book.
We look for a codimension two submanifold with trivial normal
bundle such that its complement is a fiber bundle over $S^1$.

\subsection{Setup}
Let $(M,\omega)$ be a compact symplectic manifold with integral
symplectic form. Fix $\ell\in \Z_{>0}$ and consider a
Boothby--Wang bundle $P_{M,\ell}$ associated with
$(M,\ell\omega)$.

Let $H\subset M$ be a Donaldson hypersurface Poincar\'e dual to $k [\omega]$ for some positive integer
$k\in \Z_{>0}$ \cite{Donaldson:symplectic_hypersurfaces}. The restriction of $P_{M,\ell}$ to the symplectic submanifold $H$, which we denote by $P_{H, \ell}$, is a codimension two contact submanifold in $P_{M,\ell}$.

\subsection{Neighborhood of a Donaldson hypersurface}
\label{sec:nu}

By the Weinstein symplectic neighborhood theorem, a neighborhood
of $H$ in $M$ can be identified with the normal bundle to $H$. By
Corollary~11.2 of \cite{Milnor:characteristic_classes}, the
fundamental cohomology class for the normal bundle of $H$ in $M$
corresponds to a canonical cohomology class $u'\in H^2(M,M-H;\Z)$.
Write the inclusion $(M,\emptyset)\hookrightarrow (M,M-H)$ by
$j_M$. Since the homology class $[H]$ is Poincar\'e dual to
$k[\omega]$, it follows that $u'|_M:=j_M^*u'$ is equal to
$k[\omega]$ by Problem 11-C of
\cite{Milnor:characteristic_classes}. Theorem~11.3 of
\cite{Milnor:characteristic_classes} then tells us that the first
Chern class of the normal bundle $\nu_M(H)$ of $H$ in $M$ is given
by
$$
c_1(\nu_M(H))=i^*(u'|_M)=i^*(k[\omega]),
$$
where $i:H\to M$ is the inclusion. Thus the normal bundle
$\nu_M(H)$ can be identified with the associated line bundle
$$
\nu_M(H)\cong P_{H,k}\times_{S^1} \C,
$$
where $S^1$ acts diagonally on $P_{H,k} \times \C$ by
\[
(x, v) \cdot a=(x \cdot a, va)
\] for $a \in S^1$, $x\in P_{H, k}$, and $v \in \C$. The symplectic form on
$\nu_M(H)$ can then be expressed as
\[
\Pi_k^*\,i^*(k\ow) + \frac{1}{2\pi} d(r^2\theta) - \frac{1}{2\pi}
d(r^2 d\phi),
\] where $(r, \phi)$ are the polar coordinates on $\C$, $\Pi_k: P_{H,k} \to H$ is the projection,
and $\theta$ is the connection $1$-form on $P_{H,k}$ with $d\theta = -2\pi \Pi^*_k\,i^* (k\ow)$; cf. Biran~\cite{Biran:Lagrangian_barrier}.

\begin{remark}
\label{remark:convex_end_M-H}
The hypersurface $H$ can be seen as the convex end of $M-H$.
More precisely, there is a neighborhood $\nu_M(H)$ such that $M-\nu_M(H)$ carries a compact Weinstein structure: see \cite{Giroux:ICM2002}, Proposition 11.
\end{remark}

\subsection{Choice of symplectic form}
\label{sec:k=l}
We shall now argue that we can only expect the Boothby--Wang bundle $P_{H,\ell}$ over the Donaldson hypersurface $H$ to serve as the binding for an open book on $P_{M,\ell}$ if we choose $k=\ell$.

The choice of $\ell$ dividing $k$ is motivated by the following proposition.
\begin{proposition}\label{prop:l divides k}
Suppose $\ell$ divides $k$. Then the normal bundle
$\nu_{P_{M,\ell}}( P_{H,\ell})$ of $P_{H,\ell}$ in $P_{M,\ell}$ is
trivial.
\end{proposition}
\begin{proof}
We consider the following diagram of bundles
$$
\entrymodifiers={+!!<0pt,\fontdimen22\textfont2>}
\xymatrix{
~ & S^1 \ar[d] & S^1 \ar[d] \\
\C \ar[r] & \nu_M(H)\tilde \times S^1 \ar[d] \ar[r] & P_{H,\ell} \subset P_{M,\ell} \ar[d]^\Pi \\
\C \ar[r] & \nu_M(H) \ar[r]& H\subset M.
 }
$$
The columns in this diagram represent Boothby--Wang bundles and
the rows indicate normal bundles. We use the same notation for the
projection $\Pi: P_{M,\ell} \to M$ and its restrictions.
$\nu_M(H)$ is identified with a tubular neighborhood of $H$ in
$M$.

According to the diagram, the Boothby--Wang bundle $\nu_M(H)\tilde \times S^1$ can be considered as the normal bundle of $P_{H,\ell}$ in $P_{M,\ell}$,
$$
\Pi^{-1}(\nu_M(H))=\nu_{P_{M,\ell}}(P_{H,\ell}).
$$
In order to regard $P_{H,\ell}$ as a binding of an open book for $P_{M,\ell}$, its normal bundle has to be trivial.
Let $i: H \to M$ denote the inclusion.
The diagram of Gysin sequences
\[
\entrymodifiers={+!!<0pt,\fontdimen22\textfont2>}
\xymatrix{
H^0(M;\Z) \ar[r]^{\cup \ell[\omega]} & H^2(M;\Z)  \ar[r]^{\Pi^*} \ar[d]^{i^*}
&H^2(P_{M,\ell};\Z) \ar[d]\\
H^0(H;\Z) \ar[r]^{\cup \ell i^*[\omega]} & H^2(H;\Z) \ar[r]^{\Pi^*}
&H^2(P_{H,\ell};\Z)
}
\]
shows that the first Chern class
\[
c_1( \nu_{P_{M,\ell}}(P_{H,\ell}) )=\Pi^*c_1( \nu_M(H)
)=\Pi^*i^*(k[\omega])
\]
is zero by exactness,  if $\ell$ divides $k$.
\end{proof}

\begin{remark}
The condition of Proposition~\ref{prop:l divides k} is not always necessary for the normal bundle
$\nu_{P_{M,\ell}}( P_{H,\ell})$ to be trivial.
We can take for example the case where $H$ consists of points in a surface $(M,\omega)$.

However, to obtain a proper open book, the condition $\ell$
divides $k$ is still necessary as the following example shows.
Consider $\R P^3$ as a Boothby--Wang bundle over $(S^2,\omega)$.
If $\omega$ represents a primitive cohomology class, then we have
$\ell= 2$.

If we choose $[H]$ to be Poincar\'e dual to $[\omega]$, then $[H]$
is represented by a single point. This results in a decomposition
of $\R P^3$ into two solid tori, one for a neighborhood of the
fiber over $H$, and one for the complement. The gluing map for
this pair of solid tori does not correspond to an open book,
because the projection to $S^1$ in a neighborhood of the fiber
over $H$ has the form
\begin{align*}
P_{H,\ell}\times (D^2-\{0\}) & \longrightarrow S^1 \\
(e^{i\psi},re^{i\phi}) & \longmapsto  e^{i(2\phi-\psi)},
\end{align*}
where we have identified $P_{H,\ell}$ with $S^1$.
Baker, Etnyre and Van Horn-Morris \cite{Baker:rational_open_book}
refer to such structures as rational open books.
\end{remark}

Now consider the case that $k$ divides $\ell$.

\begin{proposition}
Suppose $k$ divides $\ell$.
Then the restriction of the Boothby--Wang bundle $P_{M,\ell}$ to $M-H$ is trivial.
\end{proposition}
Note that we can think of a trivial $S^1$-bundle over $M-H$ as an $(M-H)$-bundle over $S^1$, which is necessary for an open book.
\begin{proof}
Consider the long exact sequence of the pair $(M,M-H)$ in cohomology,
\begin{equation*}
\entrymodifiers={+!!<0pt,\fontdimen22\textfont2>} \xymatrix{
H^2(M,M-H;\Z) \ar[r]^-{j_M^*} & H^2(M;\Z) \ar[r]^-{j^*} &
H^2(M-H;\Z). }
\end{equation*}
All maps are pullbacks under inclusion. As we have seen in the
beginning of Section~\ref{sec:nu} $u'|_M=k[\omega]$. Furthermore,
$u'|_M=j_M^* u'$, so $j^*(u'|_M)=0$ by exactness. Hence $P_{M,k}$
is trivial when restricted to $M-H$. As $k$ divides $\ell$, it
follows that $P_{M,\ell}$ is trivial when restricted to $M-H$.
\end{proof}

\section{Abstract open book}
\label{sec:abs open book}

Motivated by Section~\ref{sec:k=l}, we choose $k=\ell$ in our
search of an open book for a Boothby--Wang bundle associated with
$(M, \ell \ow)$ accompanied by a Donaldson hypersurface $H\subset
M$ Poincar\'e dual to $k[\ow]$. In principle, we can then try to
show directly that we obtain a contact open book in terms of the
$S^1$ bundle away from the binding. However, it is more convenient
to approach the problem by constructing an open book with fibered
Dehn twist as monodromy. We then show that the resulting contact
manifold is contactomorphic to a Boothby--Wang bundle we were
considering. For simplicity, we rescale the symplectic form and
set $k=\ell =1$.

Let $(W,-d\lambda/ 2\pi )$ be a compact Weinstein manifold such that the
boundary $(P=\partial W,\theta=\lambda|_{P})$ is a Boothby--Wang
bundle over some symplectic manifold $(H,\omega_H)$ with
projection map $\Pi_H: P \to H$. By the Boothby--Wang condition,
all Reeb orbits of $\theta$ are periodic. We denote the Reeb
vector field by $R_\theta$.

We can construct two contact manifolds out of the data given above. First of all, we can
define a symplectic manifold $M$ and a Boothby--Wang bundle over $M$. Secondly,
as discussed in Section~\ref{sec:symplectomorphisms_monodromy}, we can define a fibered Dehn twist $\tau$ for $W$ along its boundary, and then define a contact open book with page $W$ and monodromy $\tau$.

The following diagram illustrates the constructions we shall perform.
The maps will be defined subsequently. Note that the horizontal maps are only defined on
subsets of the spaces in the diagram, since they serve as gluing maps.

\begin{equation}
\label{eq:diagram_Boothby-Wang}
\entrymodifiers={+!!<0pt,\fontdimen22\textfont2>}
\xymatrix{
& \nu/ \text{Binding piece} & \text{Middle piece} & W \text{ piece}\\
\text{Open book} & P\times {\mathring D^2} \ar[r]^-{\psi_{OB}} & P\times I \times \R \,/{\sim} \ar[r]^-{\id} & W\times S^1 \\
\text{Boothby--Wang} & \txt{$(P\times_{S^1}{\mathring D^2})  \tilde \times S^1$\\ $\cong P\times {\mathring D^2}$} \ar[r]^-{\psi_{BW}} \ar[d]_{\Pi_\nu} \ar[u]^{\id} &  P\times I \times S^1 \ar[r]^-{\id} \ar[d]_{\Pi_{mid}} \ar[u]^{\psi_{mid}} & W\times S^1 \ar[d]^{\Pi_W} \ar[u]_{\id} \\
\text{Symplectic} & P\times_{S^1}{\mathring D^2} \ar[r]^-{\psi_S} & P\times I \ar[r]^-{\id }& W \\
}
\end{equation}

\subsection{Symplectic manifold}\label{sec:symplectic construct}
Let us now define the three symplectic pieces we shall patch together to form our symplectic manifold $M$. Note that the sizes we choose for the construction are artificial.

\begin{itemize}
\item The $W$ piece is the given Weinstein manifold equipped with the exact symplectic form $-\frac{1}{2\pi}d\lambda$. In a collar neighborhood of the boundary, the symplectic form looks
like a symplectization form.
For later computations it is convenient to rescale this form though.
In other words, we take
$$
(P\times I^{-},-\frac{1}{2\pi} d (e^{t-C}\theta) \,)
$$
as a collar neighborhood of the boundary of $W$ for a fixed positive constant $C$.
Here $I^{-}$ stands for the interval $] {-1}, 0]$.

\item The middle piece $P\times I$ serves as an auxiliary piece and we furnish it with the exact symplectic form
$$
-\frac{1}{2\pi}d\left( \rho(t) \theta \right),
$$
where $\rho$ is a function defined on $I$ that we shall specify later. Here $I$ stands for the interval
$] {-1}, 1 [$.
\item The last piece is the associated disk bundle $\nu:=P\times_{S^1}{\mathring D^2}$,
regarded as the
orbit space of $P\times {\mathring D^2}$ under the $S^1$ action
\[
(x,re^{i\phi})\cdot a= (x\cdot a,re^{i(\phi+a)}).
\] Here ${\mathring D^2} \subset \C$ is the open disk at $0$ of radius $1$
with polar coordinates $(r, \phi)$, and $a \in S^1 \cong \R/2\pi\Z$ is identified with $e^{ia}\in \C$.
We take the symplectic form
\[
\ow_\nu=\Pi_H^* \omega_H+\frac{1}{2\pi}d\left( r^2 \theta \right) -\frac{1}{2\pi} d\left( r^2 d\phi \right)
=
-\frac{1}{2\pi} d \left( (1-r^2)\theta -(1-r^2)d\phi \right).
\]
Note that this symplectic form is \emph{not} exact but it is an integral symplectic form on $\nu$ with the cohomology class  $\Pi_H^*[\omega_H]$.
In our conventions, the connection $1$-form $\theta$ of the Boothby--Wang bundle $P$ satisfies $d\theta=-2\pi \Pi_H^*\omega_H$.
\end{itemize}

Next we define the two gluing maps between the pieces. They ought to be
symplectomorphisms so that we obtain a closed symplectic manifold $(M, \ow)$. This imposes necessary behaviors on our function $\rho$.

We begin with gluing the middle piece $P\times I$ to $W$ using the identity:
\[
P \times I \supset P\times I^{-} \overset{\id}{\longrightarrow} P \times I^- \subset W.
\] This implies that, for $t\in I^{-}$ and small positive
values of $t$, we must have
$$
\rho(t)=e^{t-C}.
$$

On the other hand, we can glue $P\times_{S^1}{\mathring D^2}$ to $P\times I$ using the  diffeomorphism
\begin{align*}
\psi_{S}: P \times_{S^1} {\mathring D^2} \supset P\times_{S^1} \left( {\mathring D^2}-\{ 0 \} \right) & \longrightarrow P\times I \\
\left[ x,re^{i\phi} \right] & \longmapsto  \bigl(x\cdot (-\phi), 1-r\bigr).
\end{align*}
If we pull back the symplectic form $-\frac{1}{2\pi} d( \rho(t) \theta )$ under this diffeomorphism, we find
$$
\psi_{S}^*\left(  -\frac{1}{2\pi} d( \rho(t) \theta ) \right)=-\frac{1}{2\pi} d\bigl( \rho(1-r)(\theta-d\phi) \bigr),
$$
because $\mathcal L_{-\phi R_\theta} \theta=-d\phi$.
For this symplectic form to coincide with the symplectic form on $P\times_{S^1} \left( {\mathring D^2}-\{ 0 \} \right)$ near $r=0$,
we require that 
$$
\rho(1-r)=1-r^2=(1-r)\left( 2-(1-r) \right),
$$
near $r=0$. For $t$ near $1$, we set $\rho(t)=t(2-t)$.
By gluing the three pieces together, one obtains a symplectic manifold.

On the other hand, we can go back to the discussion from Section~\ref{sec:bw}.
Given an integral symplectic manifold $(M,\omega)$, and hypersurface $H$ that is Poincar\'e dual to $k[\omega]$, one can define $W:=M-\nu_M(H)$.
It is not clear that we can then apply the above construction.
We need $W$ to be Weinstein, and such that $P=\partial W$ has a Boothby-Wang type contact form.

However, for a smoothly polarized K\"{a}hler manifold $\mathcal{P}=(M^{2n},\omega,J;H)$, i.e. $(M,\omega,J)$ is a
K\"{a}hler manifold, and $H\subset M$ is a smooth and reduced complex
hypersurface whose homology class $[H]\in H_{2n-2}(M;\Z)$
represents the Poincar\'{e} dual to $k[\omega]\in H^2(M;\Z)$ for
some $k\in\N$, the symplectic manifold $(M,k\omega)$ can be reconstructed by patching the three
symplectic pieces as above (see \cite[proof of Theorem
2.6.A]{Biran:Lagrangian_barrier}).

In order to make a general statement, consider the following.
Let $(M^{2n},\omega)$ be a closed symplectic manifold with
integral symplectic form $[\omega]\in H^2(M;\Z)$, and let $H$ be a closed symplectic hypersurface, i.e.~a codimension two closed symplectic submanifold, whose homology class $[H]\in H_{2n-2}(M;\Z)$ is the Poincar\'{e} dual to $k[\omega]\in H^2(M;\Z)$ for some $k\in\N$.
\begin{definition}
If $(M^{2n},k\omega)$ can be constructed by patching three
symplectic pieces as above, then we say that $H$ is an
\textbf{adapted Donaldson hypersurface}.
\end{definition}

\begin{remark}
For a smoothly polarized K\"{a}hler manifold
$\mathcal{P}=(M^{2n},\omega,J;H)$, the complex hypersurface $H$ is an adapted Donaldson
hypersurface. As Biran points out in \cite{Biran:Lagrangian_barrier}, the symplectic hyperplane section obtained by
Donaldson's theory of symplectic hypersurfaces
\cite{Donaldson:symplectic_hypersurfaces} is probably an adapted Donaldson hypersurface.
\end{remark}

\subsection{Boothby--Wang bundle}
We now construct the Boothby--Wang bundle over the three pieces of $M$.
\begin{itemize}
\item The symplectic form on $W$ is exact, so the associated Boothby--Wang bundle $W\times S^1$ can be endowed with the contact form
$$
d\phi+\lambda.
$$
The bundle projection is the natural one:
\begin{align*}
\Pi_W: W\times S^1 & \longrightarrow  W \\
 (x,\phi) & \longmapsto  x.
\end{align*}
\item Similarly, the Boothby--Wang bundle over the middle piece $P\times I$ looks like
$$
\left( P\times I\times S^1,d\phi+ \rho(t) \theta \right)
$$
with the projection
\begin{align*}
\Pi_{mid}: P\times I \times S^1 &\longrightarrow P \times I\\
(p,t,\phi) &\longmapsto (p,t).
\end{align*}
\item By Proposition~\ref{prop:l divides k} and its proof, we can identify the Boothby--Wang bundle over $\nu$ with the manifold $P\times {\mathring D^2}$.
We furnish it with the contact form
$$
\alpha_\nu=(1-r^2)\theta+r^2d\phi.
$$
The corresponding Reeb field is given by
$$
R_\theta+\partial_\phi
$$ and therefore generates an $S^1$ action on $P\times {\mathring D^2}$. This Reeb action coincides with the $S^1$ action we used to define $\nu$ as an orbit space.
We check that the map
\begin{align*}
\Pi_\nu: P \times {\mathring D^2} & \longrightarrow \nu = P \times_{S^1} {\mathring D^2} \\
(x,v) & \longmapsto [x,v]
\end{align*}
pulls back the symplectic form $-2\pi \ow_\nu$ to $d\alpha_\nu$.
It follows that $\alpha_\nu$ is a connection $1$-form and $\Pi_\nu$ is the projection map
for this $S^1$-bundle.
Note that, as a Boothby--Wang bundle, it is not trivial.
\end{itemize}

The gluing maps are induced from the symplectic gluing maps used for $M$ as follows:
\begin{align}
 P\times I \times S^1 \supset P \times I^{-} \times S^1 &\overset{\id}{\longrightarrow} W\times S^1. \notag \\
\intertext{and}
\psi_{BW}: P\times {\mathring D^2} \supset P\times \left( {\mathring D^2}-\{ 0\} \right)
& \longrightarrow    P\times I \times S^1 \\
(x,re^{i\phi}) & \longmapsto  \bigl(x \cdot(-\phi),1-r,\phi\bigr).\notag
\end{align}\label{psi-bw}

\subsection{Contact open book}
Finally we construct a contact open book out of the three pieces announced in our diagram~\eqref{eq:diagram_Boothby-Wang}.
For the construction, we adopt a method similar to the standard one described in  Section~\ref{sec:contact_open_books}. However, we separate what used to be one page into a piece with trivial monodromy and a piece with a perturbation of a fibered Dehn twist as monodromy.
In fact, our monodromy is not the identity near the boundary, so we need to glue differently.
We shall give a recipe to correct this
in Section~\ref{sec:deforming_contact_form_near_binding}.

First we consider the following pieces:
\begin{itemize}
\item On $W$, we take the identity for the monodromy, so the mapping torus looks like
$W \times S^1$ with contact form $d\phi+\lambda$.
\item The middle piece $P \times I \times \R \,/{\sim}$ carries a nontrivial monodromy given by
\[
(x, t, \phi) \sim \bigl(x\cdot f(t), t, \phi+h(t)\bigr).
\]
By the same token as in Section \ref{sec:symplectomorphisms_monodromy}, we set
\[
h(t)=A - e^{t-C}f(t) + \int^t_0 e^{s-C}f(s)\,ds.
\] The function $f: I \to \R$ shall be specified later. Nevertheless, we
demand $f(t)=0$ for $t \in I^{-}$, and choose $A=2\pi$.
We see that $d\phi + e^{t-C}\theta$ descends to a well-defined contact form here.
\item
The neighborhood of the binding is given by $P \times {\mathring D^2}$ with contact form $h_1(r)\theta + h_2(r)d\phi$.
Since we will glue in a way that differs from the standard method for open books, we choose $h_1(r)=1-r^2$ and $h_2(r) = r^2$.
\end{itemize}

Next let us define the gluing maps. For the trivial monodromy part, we use the identity map:
\[
 P\times I \times \R \,/{\sim} \,\supset\, P\times I^{-} \times \R \,/{\sim}  \overset{\id}{\longrightarrow}  W\times S^1.
\]
To glue in the binding piece, we first define an auxiliary map
\begin{align}
\psi_{mid}: P \times I \times S^1 & \longrightarrow  P \times I \times \R \,/{\sim} \\
    (x, t, \phi) & \longmapsto  \left(x \cdot \frac{f(t)\phi}{2\pi}, t, \frac{h(t)\phi}{2\pi}\right). \notag
\end{align}\label{psi-mid}
Then we define the gluing map as composition of $\psi_{mid}$ and $\psi_{BW}$:
\begin{align}
\psi_{OB}:  P \times {\mathring D^2} \supset P\times ({\mathring D^2}-\{0\}) & \longrightarrow  P \times I \times \R\,/{\sim} \\
    (x, re^{i\phi}) & \longmapsto  \psi_{mid}\circ \psi_{BW}(x,re^{i\phi}). \notag
\end{align}\label{psi-ob}

\subsection{The twisting profile}
\label{sec:computing_profile}
We have defined all maps in the diagram~\eqref{eq:diagram_Boothby-Wang}, but two of the maps still depend on the yet to be defined twisting profile $f$.
Let us now find out what it should be.

We pull back the open book form using the diffeomorphism $\psi_{mid}$. Using a computation similar to Section~\ref{sec:symplectomorphisms_monodromy}, we see that
\[
\psi_{mid}^* (d\phi + e^{t-C}\theta) = \frac{1}{2\pi}\left(2\pi+
\int^{t}_0 e^{s-C}f(s)\,ds \right)d\phi + e^{t-C}\theta.
\]
If we choose the profile $f$ appropriately, this becomes a multiple of the Boothby--Wang form
\[
d\phi + \rho(t)\theta.
\]
In other words, we solve the equation
\[
\rho(t) = 2\pi\frac{e^{t-C}}{2\pi + \int^{t}_0 e^{s-C}f(s)\,ds}
\]
for the profile function $f$, and we obtain
\begin{equation}
\label{eq:twisting_profile}
f(t) = 2\pi\frac{\rho(t)-\rho'(t)}{\rho(t)^2}.
\end{equation}
Since the behavior of the function $\rho(t)$ for $t$ near $0$ or $1$ is determined by our choices of symplectic forms,
we see that the twisting profile $f(t)$ is $0$ for $t$ near $0$ and $f(t)\to 2\pi$ for $t\to 1$.

Hence we get a commutative
diagram~\eqref{eq:diagram_Boothby-Wang}. Furthermore, since $f$
goes from $0$ to $2\pi$, the monodromy is a right-handed fibered
Dehn twist (observe that $f$ is the twisting profile for the
inverse of a right-handed fibered Dehn twist).

\subsection{Deforming the contact form}
\label{sec:deforming_contact_form_near_binding} In this section we
adapt the contact form on the set $P\times {\mathring D^2}$ to obtain a
compatible open book. Let $f$ be a smooth monotone function which
is $0$ near $0$ and $2\pi$ near $\eta$, where $0<\eta<\min \{
C,1\}$. We now take this $f$ as twisting profile. Note that $h(t)$, the function used in the definition of the mapping torus, is always positive.

Let 
$$
\alpha_0=\psi^*_{OB}(d\phi+e^{t-C}\theta)=h_1^0(r)\theta+h_2^0(r)d\phi,
$$
where 
$$
h_1^0(r)=e^{1-r-C}, \text{ and }h_2^0(r)=1+\frac{1}{2\pi}\int^{1-r}_0
e^{s-C}f(s)\, ds-e^{1-r-C},
$$
for $r\in ]1-\eta,1[$. Note that $h_2^0(r)$ is
constant near $r=1-\eta$. We extend $h_1^0$ and $h_2^0$ near
$r=1-\eta$ such that
$h_1^{0\prime}(r)<0,h_1^0(r)>0,h_2^{0\prime}(r)\ge 0,h_2^0(r)>0$
for $r>0$ and $h_1^0(r)=1-r^2,h_2^0(r)=r^2$ near $r=0$. We obtain
a contact open book whose monodromy is a right-handed fibered Dehn
twist.

On the other hand, let
$$
\alpha_1=\psi_{BW}^*(d\phi+\rho(t)\theta)=h_1^1(r)\theta+h_2^1(r)d\phi,
$$
where $h_1^1(r)=\rho (1-r),h_2^1(r)=1-\rho (1-r)$. Note that
$$
h_1^{1\prime}(r)<0,h_1^1(r)>0,h_2^{1\prime}(r)\ge 0,h_2^1(r)>0
$$
for $r>0$ and $h_1^1(r)=1-r^2,h_2^1(r)=r^2$ near $r=0$. The
contact forms $\alpha_0$ and $\alpha_1$ are the same near $r=1$. For a contact form $h_1(r)\theta+h_2(r)d\phi$, the following condition imposed on $h_1$ and $h_2$,
\[
h_1'(r)<0,
h_1(r)>0, h_2'(r)\ge 0, h_2(r)>0
\] for $r>0$ and 
\[
h_1(r)=1-r^2,
h_2(r)=r^2
\] near $r=0$, is a convex condition. Thus we can connect
$\alpha_0$ and $\alpha_1$ by $(1-s)\alpha_0+s\alpha_1$. Then use
Gray stability to see that the associated contact structures are
contactomorphic. Hence we deform the contact form on $P\times {\mathring D}^2$ to obtain a compatible open book.

\subsection{Summary}
We summarize these results in the following theorem.
\begin{theorem}
\label{thm:Boothby_Wang_manifold_open_book}
    Let $W$ be a Weinstein domain with boundary $\partial W$ given by a Boothby--Wang bundle $P$ over $H$.
Let $\tau$ be a fibered Dehn twist on $W$ along the boundary $\partial W=P$. Then $\open(W, \tau)$
is contactomorphic to the Boothby--Wang bundle over the symplectic
manifold $(M, \omega)$ as constructed in
Section~\ref{sec:symplectic construct}.
\end{theorem}

\begin{corollary}
    Let $(M, \omega)$ be a manifold with integral symplectic form $\omega$ accompanied by an adapted Donaldson hypersurface $H$ Poincar\'e dual to $[\ow]$. Then the Boothby--Wang bundle $P_M$ associated with $(M, \ow)$ has an open book decomposition whose monodromy is a right-handed fibered Dehn twist.
\end{corollary}

\begin{proof}
This follows from Theorem~\ref{thm:Boothby_Wang_manifold_open_book}.
\end{proof}

\subsection{Boothby--Wang orbibundles over symplectic orbifolds}
\label{sec:Boothby_Wang_orbifold}
Let us now consider a multiply fibered Dehn twist as monodromy for a contact open book.
We begin by showing that the resulting contact manifold carries an $S^1$-action generated by its Reeb field.
We copy the contact part of the diagram we used earlier,

\[
\entrymodifiers={+!!<0pt,\fontdimen22\textfont2>}
\xymatrix{
\text{Open book} & P\times {\mathring D^2} \ar[r]^-{\psi_{OB, N}} & P\times I \times \R \,/{\sim} \ar[r]^-{\id} & W\times S^1.\\
\text{Boothby--Wang} &  (P\times_{S^1}{\mathring D^2})  \tilde \times S^1\cong  P \times {\mathring D^2} \ar[r]^-{\psi_{BW,N}} \ar[u]^{\id} &  P\times I \times S^1 \ar[r]^-{\id}  \ar[u]^{\psi_{mid,N}} & W\times S^1  \ar[u]^{\id} \\
}
\]

For $N\in\N$, let $W\times S^1$ be endowed with the contact form
$d\phi+\frac{1}{N}\lambda$. Over the middle piece $P\times I$, the
role of $\rho$ will be taken by $\rho_N:=\frac{1}{N} \rho$. For
the twisting profile $f_N$ we take $Nf$, where $f$ is the profile
found in Section~\ref{sec:computing_profile}. We set
$h_N(t)=h(t)$. We adjust the gluing maps as follows.
\begin{align*}
\psi_{BW,N}:    P \times {\mathring D^2} \supset P\times ({\mathring D^2}-\{0\}) & \longrightarrow  P \times I \times S^1 \\
    (x, re^{i\phi}) & \longmapsto  \bigl(x \cdot (-N \phi), 1-r, \phi\bigr),
\end{align*}
and
\begin{align*}
\psi_{mid,N}:   P \times I \times S^1 & \longrightarrow  P \times I \times \R \,/{\sim} \\
    (x, t, \phi) & \longmapsto  \left(x \cdot \frac{f_N(t)\phi}{2\pi}, t, \frac{h_N(t) \phi}{2\pi}\right).
\end{align*}

If we pull back the contact form $d\phi +\rho_N(t)\theta$ by $\psi_{BW,N}$ we find the contact form
$$
\alpha_{N}=\frac{1}{N}(1-r^2)\theta+r^2 d\phi
$$
near $r=0$.

This specifies the contact form on each of the pieces on the ``Boothby--Wang'' side.
We check that all Reeb orbits are periodic.
\begin{itemize}
\item On the binding piece $P\times {\mathring D^2}$, the Reeb field of the contact form $\alpha_{N}$ is given by
$$
R_{\alpha_{N}}=NR_\theta+\partial_\phi.
$$
It generates a locally free $S^1$ action on $P \times {\mathring D^2}$. Indeed, the $S^1$ action is given by
\[
(x, v)\cdot a = (x \cdot a^N, av).
\] We see that $(x, 0)$ is fixed by $\Z_N$, while the stabilizer for any other $(x, v)$, $v\neq 0$, is trivial.
\item On the middle piece, the Reeb field of the contact form
$d\phi+\rho_N(t)\theta$ is given by $R=\partial_\phi$. 
\item On $W\times S^1$, the Reeb field of the contact form
$d\phi+\frac{1}{N}\lambda$ is given by $R=\partial_\phi$.
\end{itemize}
These contact forms fit together to a global contact form $\alpha$
with our gluing maps, so we obtain a closed contact manifold
$(Y,\alpha)$ whose Reeb orbits are all periodic. The orbits
corresponding to the binding have period $2\pi/N$, whereas all
other orbits have period $2\pi$. In particular, this implies that
the quotient of the presymplectic manifold $(Y,d\alpha)$ by the $S^1$-action is a symplectic orbifold.
\begin{theorem}
\label{thm:Boothby_Wang_orbifold_open_book}
Let $W$ be a Weinstein domain with boundary $\partial W$ given by a Boothby--Wang bundle $P$ over $H$. Let $\tau$ be a fibered Dehn twist on $W$ along the boundary $\partial W =P$. Then $\open(W,
\tau^N)$ is contactomorphic to the Boothby--Wang orbibundle over
the symplectic orbifold $(Y,d\alpha)/S^1$.
\end{theorem}

\section{Applications}
\label{sec:applications}

We conclude this paper with some applications of the above open book decompositions and the mean Euler characteristic.
We consider certain Boothby--Wang orbibundles $P_M$ over symplectic orbifolds $M$.
By the correspondence from Theorem~\ref{thm:Boothby_Wang_orbifold_open_book} we can use contact invariants to deduce non-triviality of fibered Dehn twists.

We shall consider two cases.
Suppose $(M, \omega)$ is an integral symplectic manifold with an
adapted Donaldson hypersurface $H$ and $\tau$ is a right-handed fibered Dehn twist.
\begin{itemize}
\item If $\pi_2(M)=0$, then $\tau$ is not symplectically isotopic to the identity relative to the boundary.
This case was already considered by Biran and Giroux \cite{Biran_Giroux:fibered_Dehn}: they used Lagrangian Floer homology to prove this result; we shall give a different argument.
\item  If $c_1(M)=c [\omega]$, then the mean Euler characteristic and index-positivity give an efficient criterion to see whether fibered Dehn twists are symplectically isotopic to the identity relative to the boundary.
\end{itemize}

\subsection{Non-contractible fibers and $\pi_2(M)=0$}

\begin{theorem}[Biran and Giroux]
\label{thm:pi_2_fibered_twist_not_isotopic_id}
Let $W$ be a Weinstein domain whose boundary is a Boothby--Wang contact manifold $(P,\theta)$ over a symplectic manifold $H$.
Suppose that the integral symplectic manifold $M$, obtained via the construction in Section~\ref{sec:symplectic construct}, satisfies $\pi_2(M)=0$.
Then a right-handed fibered Dehn twist $\tau$ along $P=\partial W$ is not symplectically isotopic to the identity relative to the boundary.
\end{theorem}

\begin{remark}
Alternatively, we could take any integral symplectic manifold $(M,\omega)$ with $\pi_2(M)=0$ and find an adapted Donaldson hypersurface $H$ in $M$; its complement $W:=M-\nu(H)$ then satisfies the above condition.
\end{remark}

\begin{proof}
By Theorem~\ref{thm:Boothby_Wang_manifold_open_book} it follows that $\open(W,\tau)$ is contactomorphic to the Boothby--Wang bundle $P_M$ over $M$, whose periodic Reeb orbits are exactly the $S^1$-fibers.
The homotopy exact sequence for the fibration $S^1\to P_M\to M$ shows us that each fiber is non-contractible,
$$
0\cong \pi_2(M) \stackrel{p_*}{\longrightarrow} \pi_1(S^1) \longrightarrow \pi_1(P_M),
$$
so the condition that $\pi_2(M)=0$ implies that all Reeb orbits are non-contractible in $P_M$.

Assume that $\tau$ is symplectically isotopic to the identity relative to the boundary.
Then the following contact open books are contactomorphic
$$
\open(W,\id) \cong \open(W,\tau).
$$
By Theorem~\ref{thm:subcritical_open_book}, it follows that $P_M\cong \open(W,\tau)$ is subcritically Stein fillable.
We claim that then every contact form for the contact structure on $P_M$ must have contractible Reeb orbits.

Corollary~3 from \cite{Frauenfelder:hamiltonian_dynamics} implies
that $P_M$ has a Reeb orbit $\gamma$ that is contractible in its
subcritical filling $W\times D^2$. See also
\cite{Viterbo:Floer_applications}. To see that this orbit is also
contractible in the boundary $P_M$, we use that in our setup $\dim
P_M\geq 3$, so the filling has dimension at least $4$. Since the
subcritical filling $W\times D^2$ can be obtained from
$[0,1]\times P_M$ by attaching handles of index $\ge 3$, we see
that $$i_*:\pi_1(P_M)\longrightarrow\pi_1(W\times D^2)$$ is an
isomorphism. This gives the existence of a contractible Reeb orbit
in $P_M$, which contradicts our earlier observation that the
Boothby-Wang bundle $P_M$ does not have any periodic contractible
Reeb orbits.
\end{proof}

\subsection{Powers of fibered Dehn twists}
Next, we shall distinguish powers of fibered Dehn twists.
We need a few lemmas that all use the following setup and notation.
\begin{setupS}{Setup {\bf S}}
\label{setupS}
\begin{enumerate}
\item{} Following Section~\ref{sec:Boothby_Wang_orbifold}, construct a Boothby--Wang orbibundle by the taking an integral symplectic manifold $M$ with an adapted Donaldson hypersurface $H$ that is Poincar\'e dual to $k[\omega]$, where $[\omega]\in H^2(M;\Z)$ is primitive.
It follows that we can give $W:=M-\nu_M(H)$ a Weinstein structure.
Assume in addition that $M$ and $H$ are simply-connected, and that $\dim M=2n-2 \geq 6$.
We denote the Boothby-Wang bundle over $(H,k[\omega|_H])$ by $P$. This is also the contact boundary of $W$.
\item For a positive integer $N$, define the contact open book $(P_N,\theta_N):=\open(W,\tau^N)$: this is a Boothby--Wang orbibundle over the symplectic orbifold $M_N$. As a topological space, we have $M_N\cong P_N/S^1$.
\item the Chern class of $M_1=M$ can be written as $c_1(M)=c[\omega]$.
\end{enumerate}
\end{setupS}
We will call these assumptions setup {\bf S}.
Note that the Boothby--Wang orbibundle obtained this way satisfies conditions P1 and P2.
To see the last claim, we have the following lemma.
\begin{lemma}
The first Chern class of the contact structure in setup {\bf S} is torsion.
\end{lemma}
\begin{proof}
First consider $N=1$. Then $P_1$ is a Boothby--Wang bundle over the symplectic manifold $M$.
Consider a part of the Gysin sequence for the circle bundle $S^1\to P_1\to M$,
$$
H^0(M) \stackrel{\cup k[\omega]}{\longrightarrow} H^2(M)
\stackrel{\pi^*}{\longrightarrow} H^2(P_1).
$$
We have $\xi_1\cong \pi^* TM$, so we see that $c_1(\xi)=-\pi^*c_1(TM)=-\pi^*c[\omega]$ is torsion, since $k\neq 0$.

For $N>1$, we use the Mayer-Vietoris sequence. Put $A_N=P\times {\mathring D^2}$, and let $B_N$ be the mapping torus of $W$ with monodromy $\tau^N$.
Noting that $A_N\cap B_N \simeq P\times S^1$ we find
$$
\underset{\cong 0}{H^1(A_N)}\oplus \underset{\cong \Z}{H^1(B_N)} \stackrel{i^1}{\longrightarrow} \underset{\cong \Z}{H^1(A_N\cap B_N)} \longrightarrow
H^2(P_N) \stackrel{j^2}{\longrightarrow} H^2(A_N) \oplus H^2(B_N).
$$
The map $i^1$ is an isomorphism, so $j^2$ is injective.
Now observe that the restriction of $c_1(\xi_N)$ to both $A_N$ and $B_N$ is a torsion class.
Indeed, the contact structure over $B_N$ is a Boothby--Wang bundle for an exact, symplectic manifold, and the restriction to $A_N$, a neighborhood of the binding, has the same Chern class as in the case $N=1$.
\end{proof}

\begin{lemma}[Mean index for Boothby-Wang orbibundles]
\label{lemma:mean_index_BW_orbi}
Suppose we have the setup {\bf S} as above. Then the Maslov index of a $k$-fold cover of a principal orbit is equal to
$$
2(N(c-k)+k).
$$
\end{lemma}
\begin{proof}
In a neighborhood of $H\subset M_N$, the Boothby--Wang orbibundle looks like
$$
(P\times {\mathring D^2}, \frac{1}{N}(1-r^2)\theta+r^2d\phi),
$$
where $P$ is the Boothby-Wang bundle over $H$. The Reeb field is given by
$$
R=N R_\theta+\partial_\phi.
$$
Its flow is given by $Fl^R_t(x,z)=(x\cdot Nt,e^{it}z)$.
Now write $i:H\subset M$ for the inclusion. Observe that $P$ is an $S^1$-bundle over $H$ with Euler class $i^*k[\omega]$ and that $M-\nu(H)=W$ is Weinstein. Hence the dimension condition $\dim M\geq 6$ guarantees that the map $i_*:H_2(H)\to H_2(M)$ is surjective. This implies that $i^*[\omega]$ is primitive. It follows that $\pi_1(P)\cong \Z_k$.
As in the proof of Theorem~\ref{thm:pi_2_fibered_twist_not_isotopic_id}, the inclusion $P\to W$ induces an isomorphism on $\pi_1$.
With a Seifert-Van Kampen argument we see that $\pi_1( \open(W,\tau^N) \,)\cong \Z_k$: generators are simple exceptional orbits.
A $k$-fold cover of any periodic orbit $\gamma$ is hence contractible.

Given a trivialization $\epsilon$ of the contact structure along a capping disk of a $k$-fold cover of $\gamma$ in $P$ we construct a trivialization of the contact structure on $P\times {\mathring D^2}$ by using the additional vector fields with coordinates $(x,y)$ for the open disk ${\mathring D^2}$,
$$
X=\frac{1}{N}(1-x^2-y^2) \partial_x+y R_\theta, \quad Y=\frac{1}{N}(1-x^2-y^2) \partial_y-x R_\theta.
$$
The symplectic trivialization $\epsilon\oplus span(X,Y)$ extends over a disk spanning a $k$-fold covered orbit in $P\times {\mathring D^2}$.

With respect to this trivialization we can write down a path of symplectic matrices describing the linearized flow.
First of all, let $\gamma_k$ be a $k$-fold cover of a simple periodic Reeb orbit in $(P,\theta)$.
Let $\psi(t)$ be the matrix representation of the linearized time-$t$ flow along $\gamma_k$ with respect to the trivialization $\epsilon$.
We can then compute the linearized flow of a $k$-fold cover of a principal orbit in $P\times {\mathring D^2}$ with respect to the above trivialization.
The result is
$$
\psi_{(\gamma_k(Nt),e^{ikt }z_0)}=
\left(
\begin{array}{cc}
\psi(Nt) & 0 \\
0 & e^{ikt}
\end{array}
\right)
.
$$
We see that the Maslov index of the periodic Reeb orbit $(\gamma_k(Nt),e^{ikt }z_0)$ is given by
$$
\mu( (\gamma_k(Nt),e^{ikt }z_0),t\in[0,2\pi])=
\mu( \gamma_k(t),t\in[0,2N\pi])+2k.
$$
To compute the Maslov index, we determine the first Chern class of $H$
$$
c_1(H)=c_1(i^*TM)-c_1(\nu_M(H)\,)=(c-k)i^*[\omega].
$$
With the chosen trivialization, we apply Lemma~\ref{lemma:degree_shift_BW} to compute $\mu( \gamma_k(t),t\in[0,2N\pi])=2(c-k)N$. Hence
$$
\mu( (\gamma_k(Nt),e^{ikt }z_0),t\in[0,2\pi])=2 (c-k)N+2k.
$$
We conclude that the Maslov index of a $k$-fold cover of a principal orbit is $2(c-k)N+2k$.
\end{proof}

\begin{remark}
\label{rem:good_orbits}
We see directly from this Lemma that all principal orbits are good (that means not bad) in setup {\bf S}, as all Maslov indices of these orbits are even. Looking at the proof also shows that the exceptional orbits are good.
\end{remark}

\begin{lemma}
\label{lemma:negative_c_BW_orbi}
Suppose we have the setup {\bf S}. If $c<k$, then $\tau^N$ is not symplectically isotopic to the identity relative to the boundary.
\end{lemma}

\begin{proof}
Consider 
$$
P_N=\OB(W,\tau^N).
$$
We argue by contradiction, and suppose that $\tau^N$ is symplectically isotopic to the identity relative to the boundary.
Then $P_N$ is subcritically fillable by $W\times D^2$.
It follows that the universal cover, $\tilde P_N$, is subcritically fillable by $\tilde W \times D^2$. The first Chern class of $\tilde W \times D^2$ is torsion, so Proposition~\ref{proposition:mean_euler_subcritical} tells us that
$$
SH^{S^1,+}_*(\tilde W\times D^2)\cong H_{*+n-1}(\tilde W,\partial \tilde W;\Q ) \otimes H_*(\C P^\infty ;\Q),
$$
which is index-positive and has generators in arbitrarily large, positive degree.

On the other hand, $P_N$, and therefore $\tilde P_N$ has periodic Reeb flow.
Also, the conditions P1, P2 and P4 are satisfied for $\tilde P_N$. Furthermore, P5 holds since we are assuming that $\tilde W \times D^2$ is subcritical.

Hence there is a Morse--Bott spectral sequence converging to $SH^{S^1,+}_*(\tilde W\times D^2)$, see also the proof of Proposition~\ref{prop:mean_euler_S^1-orbibundle}.
Its $E^1$-page is given by
$$
E^1_{pq}=\bigoplus_{\substack{ N_T \text{ consists of contractible orbits} \\ \mu(N_T)-\frac{1}{2}\dim (N_T/S^1)=p}} H^{S^1}_{q}(N_T;\Q).
$$
For sufficiently large $N$, all Maslov indices of covers of principal orbits are negative by Lemma~\ref{lemma:mean_index_BW_orbi}.
It follows that all Maslov indices are bounded from above, and therefore the entries of this spectral sequence have also an upper bound on their degree.
This contradicts that $SH^{S^1,+}_*(\tilde W\times D^2)$ has generators in arbitrarily large, positive degree.
We conclude that $\tau^N$ is not symplectically isotopic to the identity relative to the boundary for large values of $N$.

To obtain the claim for small $N$, we just observe that if $\tau^{N_0}$ is symplectically isotopic to the identity relative to the boundary, then so is $\tau^{N_0m}$ for any positive integer $m$.
\end{proof}

\begin{lemma}
\label{lemma:covering}
Let $(P^{2n-1},\alpha)=(P_N,\vartheta_N)$ be a cooriented contact manifold as constructed in setup~{\bf S} such that conditions P1, P2 and P3 hold, and suppose that $\pi:(\tilde P,\tilde \alpha) \to (P,\alpha)$ is a connected $m$-fold cover such that conditions P1, P2, P3, P4 and P5 hold for $(\tilde P,\tilde \alpha)$. Denote the exact filling of $\tilde P$ by $\tilde W$. Then
$$
\chi_m(SH^{S^1,+}_*(\tilde W) \, )=\chi_m(\tilde P,\tilde \alpha)=
(-1)^{n+1} \frac{\left( \frac{N}{\ell}-\gcd(N,m) \right) \chi(H)+\gcd(N,m)\chi(M) }{|\mu_P|}
$$
with $\ell=\gcd(N,k)$.
\end{lemma}

\begin{proof}
By Proposition~\ref{prop:mean_euler_S^1-orbibundle} it suffices to show $\chi_m(\tilde P,\tilde \alpha)$ satisfies the given formula. 
Write $p:P\to M_N$ and $\tilde p:\tilde P\to \tilde M_N$ for the projections to the quotient spaces.
Denote the simple exceptional orbits in $P$ by $N_{T_1}$, and the principal orbits by $N_{T_2}$.
Similarly, write $N_{\tilde T_1}$ and $N_{\tilde T_2}$ for the exceptional and principal orbits in $\tilde P$.
We first relate the periods. 
As $\tilde P$ is an $m$-fold cover, it follows that $\tilde T_1=m T_1$.
For $P$, we have $T_2=N T_1$.
In $\tilde P$, we have $\tilde T_2=\tilde N \tilde T_1$ with $\tilde N=\frac{N}{\gcd(N,m)}$
Recall that $\ell$ is defined by $\ell=\gcd(N,k)$.
The corresponding notion in $\tilde P$ is $\tilde \ell=\gcd(\tilde N,\frac{k}{m})$.
Hence we have
$$
\tilde \ell=\gcd(\frac{N}{\gcd(N,m)},\frac{k}{m})=\frac{\gcd(N,k)}{\gcd(N,m)}.
$$

To compute the equivariant Euler characteristics, we use \cite[Lemma~5.3]{Ginzburg:group}, which asserts that for an $S^1$-manifold $N$ with locally free (i.e.~only finite isotropy groups) action, one has $H_*^{S^1}(N;\Q)\cong H_*(N/S^1;\Q)$.
The exceptional orbits in $P$ and in $\tilde P$ lie both in an $S^1$-bundle over $H$.
Hence $H^{S^1}_*(N_{T_1};\Q)\cong H_*(H;\Q) \cong H^{S^1}_*(N_{\tilde T_1};\Q)$ and in particular $\chi^{S^1}(N_{\tilde T_1})=\chi(H)$.
For $\chi^{S^1}(N_{\tilde T_2})$ we decompose $M_N=\nu(H)\cup C$ and $\tilde M_N=\tilde \nu(H)\cup \tilde C$, where $\nu(H)$ is a neighborhood of $H$ in $M_N$, $C$ is the complement of $H$ in $M_N$, $\tilde \nu(H)$ is a neighborhood of $H$ in $\tilde M_N$, and $\tilde C$ is the complement of $H$ in $\tilde M_N$.
Away from the exceptional orbits, we have free circle actions: $p:P-N_{T_1}\to C$, and $\tilde p:\tilde P-N_{\tilde T_1}\to \tilde C$ are circle bundles.

We see that a point in $C$ lifts to a single orbit $\gamma$ in $P$. The preimage in $\tilde P$ under $\pi$ consists of $\gcd(N,m)$ distinct orbits, which project down to $\gcd(N,m)$ distinct points in $\tilde C$.
We apply this observation to a simplicial decomposition of $C$. It follows that each simplex in $C$ gives rise to $\gcd(N,m)$ distinct simplices in $\tilde C$.
By putting together all simplices obtained this way we obtain a simplicial decomposition for $\tilde C$. It follows that $\chi(\tilde C)=\gcd(N,m) \chi(C)$.
We conclude that 
$$
\chi^{S^1}(N_{\tilde T_2})=\chi(\tilde M_N)=\chi(H)+\gcd(N,m)\chi(C)=(1-\gcd(N,m) \, )\chi(H)+\gcd(N,m)\chi(M).
$$

Finally observe that $\mu_{\tilde P}=\mu_{P}$.
Indeed, the smallest contractible cover of a principal orbit in $P$ lifts to a contractible loop in $\tilde P$: by lifting the trivialization of the contact structure as well, we see that the Maslov indices must coincide.

Put the above into Formula~\eqref{eq:definition_mean_euler_contact}. We find
\[
\begin{split}
\chi_m(\tilde P,\tilde \alpha) &=
(-1)^{n+1} \frac{\left( \frac{\tilde N}{\tilde \ell}-1 \right) \chi^{S^1}(N_{\tilde T_1})+\chi^{S^1}(N_{\tilde T_2}) }{|\mu_{\tilde P}|}\\
&=
(-1)^{n+1} \frac{\left( \frac{N}{\ell}-\gcd(N,m) \right) \chi(H)+\gcd(N,m)\chi(M) }{|\mu_P|}.
\end{split}
\]
\end{proof}

\begin{remark}
Lemma~\ref{lemma:mean_index_BW_orbi} does not directly apply to compute the Maslov index of a $k$-fold cover of a principal orbit in $(\tilde P,\tilde \alpha)$ since
$$
\tilde \alpha=\frac{m}{N}(1-r^2)\theta+r^2d\phi,
$$
and $\frac{m}{N}$ is not necessarily an integer.
\end{remark}

A special case worth mentioning is the following.
\begin{proposition}
\label{prop:mean_euler_boothby_wang_orbifold}
Again, suppose that $(P_N^{2n-1},\theta_N)$ is as in setup {\bf S} with $k$ odd. Suppose that $(P_N,\theta_N)$ has an exact filling $W'$ such that
\begin{itemize}
\item $c_1(W')$ is torsion.
\item $i:P_N\to W'$ induces an injection on $\pi_1$.
\end{itemize}
Suppose furthermore that $N(c-k)+k\neq 0$.
Then the mean Euler characteristic of $SH^{S^1,+}(W')$ in the class of contractible orbits is
\begin{equation}
\chi_m(SH_*^{S^1,+}(W'))=(-1)^{n+1}\frac{(N-\ell )\chi(H)+\ell \chi(M)}{2| N(c-k)+k|}.
\end{equation}
with $\ell=\gcd(N,k)$.
Furthermore for $k=1$, we can rewrite this as 
\begin{equation}
\chi_m(SH_*^{S^1,+}(W'))=(-1)^{n+1}\frac{\chi(\vert \w M_N \vert)}{2 N\vert \langle c_1^{orb}(M_N) , [B_N] \rangle \vert},
\end{equation}
where 
\begin{itemize}
\item $N$ can be identified with the total number of sectors, and
\item the homology class $[B_N]$ is represented by a $2$-sphere $B_N$ lying in $\nu_{M_N}(H)$, such that
$$
\langle j^*[\omega],\pi_*([B_N])\rangle=1,
$$
where $j$ denotes the inclusion $H\subset M_N$, and $\pi: \nu_{M_N}(H)\to H$ the projection.
\end{itemize}
\end{proposition}

\begin{remark}
The simplest case of such an exact filling is a Stein filling. Our dimension assumptions show that $i^*:H^2(W')\to H^2(P_N)$ is injective. Hence $c_1(W')$ is torsion. Furthermore, our dimension assumptions imply that the inclusion $i:P_N\to W'$ induces an isomorphism on $\pi_1$.
\end{remark}

\begin{proof}
Consider the smallest contractible cover of a principal orbit: this is a $\frac{k}{\ell}$-times cover of a principal orbit.
By Lemma~\ref{lemma:mean_index_BW_orbi} we find $\mu_P=\frac{2( N(c-k)+k)}{\ell}$. Hence P3 holds.
The given conditions imply that P1, P2, P4 (use Lemma~\ref{lemma:H^1_trivial}) and P5 hold as well, so with $m=1$ we apply Lemma~\ref{lemma:covering} and obtain.
$$
\chi_m( SH^{S^1,+}_*(W')\, ) =(-1)^{n+1}
\frac{(\frac{N}{\ell}-1)\chi(H)+\chi(M)}{|\mu_P|}
$$
Combine to obtain the first claim.

We proceed to give some details for the last part.
By \cite[Corollary~3.17]{AdemLeidaRuan:book}, we have that $\chi_{orb}(M_N)=\chi(\vert \w M_N \vert)$, where $\w M_N$ is the inertia orbifold associated with $M_N$. Together with \cite[Theorem~3.17]{AdemLeidaRuan:book}, we find $\chi(\vert \w M_N\vert)=(N-1)\chi(H)+\chi(M)$.

For the Chern number, we first consider $M_1=M$. We construct a $2$-sphere $B_1$ with $[B_1] \cdot [H]=1$.
Define
$$
B_1=D_1 \cup_\partial D_2.
$$
Here $D_2$ is a disk of the form
$$
D_2=\{ [p,z]_1 \in P\times_{S^1,1} D^2_\epsilon ~|~\vert z\vert<\epsilon \}
,
$$
where $[p,z]_1$ denotes the equivalence class of the relation $(p,z)\sim_{S^1,1} (p\cdot g,g z)$.
The boundary of $D_2$ is a circle lying in $\partial \nu_M(H)\cong P$, which is simply-connected. Hence we find a disk $D_1\subset P$ bounding the same circle. 
Denote the inclusion of $B_1$ into $M_1$ by $i_1$.

Using a metric, we can split the tangent bundle $TM$ along $H$ as $TM|_{H}\cong TH\oplus \nu_{M_1}^b(H)$, where $\nu_{M_1}^b(H)$ is the normal bundle of $H$ in $M$.
Since $B_1\subset \nu_{M_1}^b(H)$, we have $i_1^*TM\cong i_1^*TH \oplus i_1^* \nu_{M_1}^b(H)$.
We are interested in $c_1(TM)=c_1(\Lambda^{top} TM)$, and we shall compute this using Chern-Weil theory.

By the above, we have $i_1^*\Lambda^{top} TM\cong i_1^*\Lambda^{top}TH\otimes i_1^* \nu_{M_1}^b(H)$.
We construct connections on $L_{1,N=1}:=i_1^*\nu_{M_1}^b(H)$, and on $L_2:=i_1^*\Lambda^{top}TH$.
\begin{itemize}
\item We trivialize the normal bundle $\nu_{M_1}^b(H)$ on the collar neighborhood $\nu_\partial(D_2)$ of the boundary of $D_2$ by $([p,z],\lambda) \mapsto ([p,z],\lambda z)$.
Here we use that $D_2$ is a disk that is normal to $H$.
We extend this trivialization over $D_1$.

Now choose a connection $\nabla_{L_{1,1}}$ that equals the trivial connection $d$ on $D_1$ and on the collar neighborhood $\nu_\partial(D_2)$.
A standard formula for the change of frame $v\mapsto \frac{1}{z}v$ gives the connection form on $D_2$: this is $z d\frac{1}{z}=-\frac{dz}{z}$.
The resulting connection is invariant under the $\Z_N$-action by $N$-th roots of unity in the disk $D_2$ near $H$.

\item For $L_2$ choose a connection $\nabla_{L_2}$ that equals the trivial connection $d$ on $D_2$.
\end{itemize}
Using the Chern-Weil construction, we have $\int_{D_2} c_1(\nabla_{L_{1,1}})=1$ since $c_1(\nu(H)\, )=[\omega]$. Alternatively, we can integrate directly.
Furthermore, $\int_{D_1} c_1(\nabla_{L_2})=c-1$ as $c_1(i^*TH)+c_1(i^*\nu_{M}^b(H)\,)=ci^*[\omega]$.
In trivializing charts we can define a connection for $L_{1,1}\otimes L_2 \cong i_1^* \Lambda^{top}TM$ by putting $\nabla_{L_{1,1}\otimes L_2}=d+\theta_{L_{1,1}}+\theta_{L_2}$, where $\theta_{L_{1,1}}$ and $\theta_{L_2}$ are the connection forms with respect to a frame for $L_{1,1}$ and $L_2$, respectively. We shall use these connections to construct a connection for the general case.

For the case $M_N$ with $N>1$, the sphere $B_N$ is replaced by the orbisphere
$$
B_N=D_1 \cup_\partial D_2^N,
$$
where $D_2^N$ is the orbidisk
$$
D_2^N=\{ [p,z]_N \in P\times_{S^1,N} D^2_\epsilon ~|~\vert z\vert<\epsilon \}
,
$$
and $[p,z]_N$ denotes the equivalence class of the relation $(p,z)\sim_{S^1,N} (p\cdot g^N,g z)$.
Note that $p$ is an orbifold point with isotropy group $\Z_N$ in $B_N$, and that $[B_N]$ satisfies the homological condition given in the Proposition.

Let $i_N$ denote the inclusion $B_N$ into $M_N$, and consider the orbibundle $i_N^*\Lambda^{top}TM_N$.
On the disk $D_1$ this is a vector bundle and the disk $D_2$ serves as a uniformizing chart for $D_2^N$, so we apply the construction of a connection to these (uniformizing) disks and find
$$
\int_{B_N}i^*_N c_1^{orb}(\Lambda^{top}TM_N)=\int_{D_1} c_1(\nabla_{L_2}) +\frac{1}{N} \int_{D_2} c_1(\nabla_{L_{1,1}})=c-1+\frac{1}{N}.
$$
\end{proof}

\begin{theorem}
\label{thm:distinguishing_powers}
Let $(M^{2n-2},\omega)$ be a simply connected symplectic manifold of dimension at least $6$ such that $[\omega]\in H^2(M;\Z)$ is a primitive element.
Suppose that $c_1(M)=c[\omega]$, and let $H$ be an adapted Donaldson hypersurface Poincar\'e dual to $k[\omega]$.
Let $\tau$ denote a right-handed fibered Dehn twist along the boundary of $M-\nu(H)$.
If $\tau^N$ is symplectically isotopic to the identity relative to the boundary, then one of the following conditions must hold,
\begin{itemize}
\item $c\geq k$, $k$ does not divide $N$, and $\chi(H)=\chi(M)=0$.
\item $c=k$, $k$ divides $N$, and $\chi(H)=0$.
\item $c> k$, $k$ divides $N$, and $\left((c-k)k+1 \right) \chi(H)=(c-k)k\chi(M)$.
\end{itemize} 
\end{theorem}

\begin{remark}
This means in many cases that all positive powers of fibered Dehn twists along the boundary of $M-\nu(H)$ are distinct.
Indeed, note that if $\tau^M$ is symplectically isotopic to $\tau^N$ relative to the boundary with $M>N$, then $\tau^{M-N}$ is symplectically isotopic to the identity relative to the boundary.
\end{remark}

\begin{proof}
By Lemma~\ref{lemma:negative_c_BW_orbi} a fibered Dehn twist cannot be symplectically isotopic to the identity relative to the boundary if $c<k$.

For $c\geq k$ we investigate the mean Euler characteristic. Take $N\in \N$ such that $\tau^{N}$ is symplectically isotopic to the identity relative to the boundary. Then for $m\in \N$, $\tau^{Nm}$ is also symplectically isotopic to the identity relative to the boundary.
Put $W:=M-\nu(H)$.
Consider $P_{Nm}=\OB(W,\tau^{Nm})$.
Then $P_{Nm}$ is subcritically fillable by $W\times D^2$.
The universal cover of $P_{Nm}$, denoted by $\widetilde{P_{Nm}}$ is then subcritically fillable by $\widetilde W \times D^2$.
Then Proposition~\ref{proposition:mean_euler_subcritical} shows that
$$
\chi_m(\widetilde W \times D^2)=(-1)^{n+1}\frac{\chi( \widetilde W)}{2}=(-1)^{n+1}\frac{ k \chi(W)}{2}=(-1)^{n+1}\frac{ k \left( \chi(M)-\chi(H) \right)}{2}.
$$
On the other hand, Remark~\ref{rem:good_orbits} and the fact that $\tilde W \times D^2$ is a subcritical Stein filling show that Lemma~\ref{lemma:covering} applies. We obtain
$$
\chi_m(\widetilde W \times D^2)=(-1)^{n+1}
\frac{(\frac{Nm}{\gcd(Nm,k)}-\gcd(Nm,k)\,)\chi(H)+\gcd(Nm,k)\chi(M)}{|\mu_P|}
$$
with $\mu_P=2(Nm(c-k)+k)/\gcd(Nm,k)$.
Comparing the two formulas for the mean Euler characteristic yields the following equation,
$$
(Nm-\gcd(Nm,k)^2\,)\chi(H)+\gcd(Nm,k)^2\chi(M)=(\, (c-k)Nm+k)k\chi(M)-(\, (c-k)Nm+k)k\chi(H),
$$
which we rewrite into
$$
(\,(k(c-k)+1)Nm+k^2-\gcd(Nm,k)^2\,)\chi(H)=
(\,k(c-k)Nm+k^2-\gcd(Nm,k)^2\,)\chi(M).
$$
We check when this equation can hold.
\begin{itemize}
\item if $k$ divides $N$, and $c=k$, then this equation reduces to $Nm\chi(H)=0$, so we conclude that $\chi(H)=0$.
\item if $k$ divides $N$, and $c>k$, then this equation reduces to $(k(c-k)+1)Nm\chi(H)=k(c-k)Nm\chi(M)$, so we conclude that $(k(c-k)+1)\chi(H)=k(c-k)\chi(M)$.
\item if $k$ does not divide $N$, then we define the functions
\[
\begin{split}
f(m)&:=(\,(k(c-k)+1)Nm+k^2-\gcd(Nm,k)^2\,)\chi(H)\\
g(m)&:=(\,k(c-k)Nm+k^2-\gcd(Nm,k)^2\,)\chi(M).
\end{split}
\]
The above equation tells us that $f(m)=g(m)$. This cannot hold for different values of $m$ with $\gcd(Nm,k)=\gcd(N,k)$ unless $\chi(H)=\chi(M)=0$. 
\end{itemize}

\end{proof}

We conclude this paper by giving some examples where  Theorem~\ref{thm:distinguishing_powers} applies.

\begin{example}
\label{example:CP-H_d}
    Let $M=\C P^{n-1}$ with $n\geq 4$ and $H=H_k$ a hypersurface of degree $k$ in $M$. One can check that
    \[
    \chi(M) = n, \quad \chi(H)=\frac{1}{k}\left( (1-k)^{n}-1 \right) +n, \quad c=n.
    \] 
Then a right-handed fibered Dehn twist $\tau$ along the boundary of $M-\nu(H)$ is not symplectically isotopic to the identity relative to the boundary unless $k=1$. Note that $\C P^n$ does not contain Lagrangian spheres. Therefore these fibered Dehn twists are not Dehn twists.
\end{example}

Note that Example~\ref{example:CP-H_d} satisfies the conditions of Theorem~\ref{thm:distinguishing_powers}.
\begin{example}
\label{example:degree_d_hypersurface}
Consider the degree $d$ hypersurface $H_d^{n-1}\subset \C P^{n}$ defined by
$$
H_d^{n-1}=\{ (z_0:\ldots:z_{n})\in \C P^{n}  ~|~\sum_j z_j^d=0 \}
.
$$
One can check that
$$
\chi(H_d^{n-1})=\frac{1}{d}\left( (1-d)^{n+1}-1 \right) +n+1
$$
and that $c$, as defined above, is equal to
$$
c=n+1-d.
$$
Take the hypersurface in $H_d^{n-1}$ given by
$$
H_d^{n-2}=\{ (z_0:\ldots:z_{n-1}:z_{n})\in H_d^{n-1} ~|~z_{n}=0 \}
.
$$
Observe that $H:=H_d^{n-2}$ is a hypersurface of degree $k=1$ in $M:=H_d^{n-1}$.
If $n>3$, then the manifolds $H_d^{n-2}$ and $H_d^{n-1}$ are simply-connected.
We apply Theorem~\ref{thm:distinguishing_powers} and check whether the last condition $\chi(H)c=\chi(M)(c-1)$ holds.
This leads to the equation
$$
\left( \frac{1}{d}\left( (1-d)^{n}-1 \right) +n \right) (n+1-d)=\left( \frac{1}{d}\left( (1-d)^{n+1}-1 \right) +n+1 \right) (n-d).
$$
Consider
$$
f_n(d):=d\cdot \left( 
\left( \frac{1}{d}\left( (1-d)^{n}-1 \right) +n \right) (n+1-d)-\left( \frac{1}{d}\left( (1-d)^{n+1}-1 \right) +n+1 \right) (n-d)
\right)
.
$$
We can simplify $f_n(d)$ to
$$
f_n(d)=(1-d)^n(1+nd-d^2)-(1-d^2).
$$
We claim that for integers $d$ with $2\leq d\leq n$, the number $f_n(d)$ does not vanish.
For $d\geq n+1$, we find $c=n+1-d\leq 0$, so we conclude:\\
\noindent{\bf Result:} If $d\geq 2$, then all powers of fibered Dehn twists along $H_d^{n-1}-\nu(H_d^{n-2})$ are pairwise distinct.

To verify our claim, we do a little computation.
First of all, note that $n> 3$ and
$$
f_n(2)=3+(-1)^n(2n-3).
$$
So $f_n(2)$ is positive if $n$ is even, and negative if $n$ is odd.
Now we check that the function $f_n$ is monotone on the interval $[2,n-2]$.
We compute
$$
f_n'(d)=(1-d)^{n-1}\left( d\left( (n+2)d-(n^2+n+2) \right) \right)+2d.
$$
If $n$ is even, then $f_n'(d)>0$ on the interval $[2,n-2]$.
If $n$ is odd, then $f_n'(d)<0$ on the interval $[2,n-2]$.

Finally, we check $f_n(n-1)$ and $f_n(n)$ separately:
$$
f_n(n-1)=\left( (2-n)^{n-1}-1 \right)(2-n)n \neq 0
$$
and
$$
f_n(n)=(1-n)^n-1+n^2\neq 0.
$$
\end{example}

\subsection{Fibered Dehn twists that are not smoothly isotopic to the identity}
The most interesting case is probably when a fibered Dehn twist is smoothly isotopic to the identity relative to the boundary, yet not symplectically.
This problem is unfortunately very hard to solve in general.
We give some examples to illustrate this. 
These examples also show that fibered Dehn twists are very often not even smoothly isotopic to the identity relative to the boundary.

\subsubsection{Dehn twists versus fibered Dehn twists}
Consider $W:=T^*_{\leq 1}S^n=\{ (q,p) \in T^*S^n ~|~\Vert p \Vert \leq 1\}$. Its boundary $P=S T^*S^n$ has a periodic Reeb flow, and this can be used to define Dehn twists and fibered Dehn twists. To define a Dehn twist, choose a smooth function $\tilde f:[0,1]\to \R$ such that $\tilde f$ is $2\pi$ near $0$ and $\tilde f$ is equal to $\pi$ near $1$. Define
\[
\begin{split}
\tau: ST^*S^n \times [0,1] & \longrightarrow ST^*S^n \times [0,1] \\
(x,t) & \longmapsto (-\id \circ Fl^R_{\tilde f(t)}(x),t)
\end{split}
\]
This defines a Dehn twist on a collar neighborhood of the boundary of $W$.
We extend the map to $-\id$ on the interior of $W$.
Note that the square of a Dehn twist, $\tau^2$, is symplectically isotopic to a fibered Dehn twist $\tau_f$.
From \cite[Theorem 1.21]{Avdek:Liouville} we have
\begin{proposition}
Fibered Dehn twists in $T^*_{\leq 1}S^n$ are not smoothly isotopic to the identity relative to the boundary unless $n=2,6$. On the other hand, fibered Dehn twists in $T^*_{\leq 1}S^2$ and $T^*_{\leq 1}S^6$ are smoothly isotopic to the identity relative to the boundary.
\end{proposition}
In particular, we see that fibered Dehn twists are often not smoothly isotopic to the identity relative to the boundary.
We give another example to describe another method to see that fibered Dehn twists are not smoothly isotopic to the identity.

\subsubsection{``Roots'' of fibered Dehn twists via coverings}
Here is a sample statement that can be obtained via coverings.
\begin{proposition}
Let $W_d:=\C P^n-\nu(H^{n-1}_d)$ be the complement of a neighborhood of a smooth hypersurface of degree $d>1$ in $\C P^n$.
Then a fibered Dehn twist in $W_d$ is not smoothly isotopic to the identity relative to the boundary. 
\end{proposition}

\begin{proof}
We start by giving another description of $W_d$.
We may assume that the smooth hypersurface of degree $d>1$ in $\C P^n$ is in standard form,
$$
H_d^{n-1}=\{ [z_0:\ldots:z_{n}]~|~\sum_i z_i^d=0 \}
.
$$
We claim that $W_d\cong V_d/\Z_d$.
Here $V_d$ is a smooth, affine variety given by
$$
V_d=\{ (z_0,\ldots,z_{n})~|~\sum_i z_i^d=1 \}
.
$$
Furthermore, we have an action of $\Z_d$ via multiplication of all coordinates with $\zeta_d$, a $d-th$ root of unity.
To see that this holds, consider the map
\[
\begin{split}
\phi: V_d/\Z_d & \longrightarrow \C P^n-H_d^{n-1} \\
[(z_0,\ldots,z_n)]_d & \longmapsto [z_0:\ldots:z_n]
\end{split}
\]
Here $[(z_0,\ldots,z_n)]_d$ denotes the equivalence class of $(z_0,\ldots,z_n)$ in $V_d/\Z_d$.

\noindent
{\bf Claim: }$\phi$ is a diffeomorphism, and in fact a biholomorphism.
It is not difficult to check, but we omit the proof here.

Now consider the contact open book $\OB(W_d,\tau_{f,W_d})$
$$
\OB(W_d,\tau_{f,W_d})\cong (L(d)=S^{2n+1}/\Z_d,\xi_0).
$$
Here the contact structure on the lens space $(L(d),\xi_0)$ is obtained by taking the quotient of $(S^{2n+1},\xi_0)$ under action by multiplication with roots of unity in each coordinate.
By taking the $d$-fold cover of the open book, we obtain a contact open book for $(S^{2n+1},\xi_0)$.
With our earlier identification $W_d\cong V_d/\Z_d$, we find 
$$
(S^{2n+1},\xi_0)=\OB(\widetilde{W_d},\tilde \tau_{f,W_d})=\OB(V_d,\tilde \tau_{f,W_d}).
$$
The cover of the monodromy $\tau_{f,W_d}$ is a map that is the identity on the boundary, and multiplication by a $d$-th root of unity in the interior. In a neighborhood of the boundary, an interpolation similar to a fibered Dehn twist occurs, with the angle going from $2\pi/d$ to $0$ instead.

Suppose now that a fibered Dehn twist $\tau_{f,W_d}$ on $W_d$ is smoothly isotopic to the identity relative to the boundary.
Then the isotopy can be lifted to its cover. Since $\tau_{f,W_d}=\id$ near the boundary of $W_d$, this remains true on the cover. It follows that the lifted monodromy, $\tilde \tau_{f,W_d}$, is smoothly isotopic to the identity near the boundary.

To obtain a contradiction, we consider two cases.
For $d=2$, we observe that $\tilde \tau_{f,W_2}$ is a standard right-handed Dehn twist on $T^*S^n$.
It is well-known that a standard right-handed Dehn twist on $T^*S^n$ is not smoothly isotopic to the identity relative to the boundary. This is a contradiction.
For $d>2$, we claim that $\tilde \tau_{f,W_d}$ acts non-trivially on homology. An easy way to see this, is to use the basis of homology given by \cite[Chapter 12]{HM:exotic}: multiplication by a $d$-th root of unity acts obviously non-trivially on this basis.
So we get a contradiction in this case as well and we conclude that $\tau_{f,W_d}$ is not smoothly isotopic to the identity relative to the boundary if $d>1$.
\end{proof}

In principle, this method can be applied in other situations as well,
such as Weinstein manifolds that are formed as the complement in an
integral symplectic manifold of an adapted Donaldson hypersurface.


\end{document}